\documentclass[twoside]{article}

\usepackage[preprint]{aistats2027}
\usepackage[round]{natbib}

\usepackage{amsmath}
\usepackage{amsthm}
\usepackage{amssymb}
\usepackage{dsfont}

\usepackage{graphicx}

\usepackage{algorithm}
\usepackage{hyperref}

\usepackage{booktabs}

\theoremstyle{plain}
\newtheorem{theorem}{Theorem}[section]
\newtheorem{proposition}[theorem]{Proposition}
\newtheorem{lemma}[theorem]{Lemma}

\theoremstyle{definition}
\newtheorem{definition}[theorem]{Definition}

\theoremstyle{remark}
\newtheorem{remark}[theorem]{Remark}

\newcommand{\indep}{\perp\!\!\!\perp}
\newcommand{\e}[1]{\mathbb{E}\left[#1\right]  } 
\newcommand{\ind}{\mathds{1}}

\usepackage[textwidth=3cm]{todonotes}

\usepackage{tikz}
\usetikzlibrary{arrows.meta}
\usetikzlibrary{calc}

\usepackage{multirow, graphicx}
\usepackage{booktabs}

\begin{document}

%

%

\twocolumn[

\aistatstitle{Sequential Conditional Independence Testing with Machine Learning Models}

\aistatsauthor{ Angel Reyero-Lobo \And Michele Meziu \And  Sebastian Uriel Arias \And Peter Grünwald }

\aistatsaddress{
Université de Toulouse \\ Inria Paris-Saclay
\And CWI
\And CWI
\And CWI \\ Leiden University
} ]

\begin{abstract}
  Conditional independence testing is a ubiquitous problem in scientific discovery.
The widely employed model-X assumption shifts the modelling burden from the dependence of the output on the inputs to the dependencies within the inputs.
Log-optimal e-variables have been studied in this setting, but it remains unclear how to incorporate machine learning models into their design.
Other approaches test exchangeability directly, yielding an e-variable with lower power in theory but, surprisingly, higher power in practice. We explain this phenomenon by decomposing the error  into \textit{null enlargement}, \textit{approximation}, and \textit{estimation} error.
The decomposition shows that GRO e-variable estimates can be beaten because of their worse approximation and estimation errors, and we explore intermediate null hypotheses between model-X conditional independence and exchangeability to reduce these errors.
Moreover, the model-X assumption often only holds  up to an estimation error, invalidating exact type-I error guarantees. We provide estimation
error bounds that accommodate \textit{triple} robustness results, achieving fast convergence rates.
\end{abstract}

\section{Introduction}
There has been a growing interest in recent years in \textit{conditional independence testing}, notably in the context of scientific discovery. A representative example is identifying which genes are associated with a disease. Performing such tests prevents spurious discoveries caused by high correlations with features that are themselves important.
%
%
Formally, we test whether a given input $X\in \mathcal{X}\subseteq\mathbb{R}$ is independent of an output $Y\in \mathcal{Y}\subseteq \mathbb{R}$ given the remaining inputs $Z\in\mathcal{Z}\subseteq\mathbb{R}^d$, namely $X \indep Y \mid Z$.
This assesses whether the input $X$ provides information about the output that is not contained in the others. If so, the feature is considered \textit{important}. The set of such important features, namely the Markov blanket, is well defined and unique whenever no input is an exact function of the others \citep{candes2018panning, LocoVSShapley}.

Moreover, while traditional statistics has mainly focused on modeling the output as a function of the inputs, as in generalized linear models, \citet{candes2018panning} introduced the \textit{model-X} assumption, which shifts the burden of modeling toward the dependence structure among the inputs themselves. This arguably milder assumption relies on the idea that the relationships among the inputs are often easier to model than their relationship with the output. For example, the inputs may be measurements collected from a common device and linked to a more complex phenomenon, as in the gene--disease setting. Alternatively, labeled data may be expensive to obtain, while unlabeled data is available for studying the inputs dependence structure.

Furthermore, in related areas such as Interpretable Machine Learning, where the aim is to identify the \textit{important} features contributing to the predictive accuracy of a model, several works have empirically and theoretically supported this assumption. Indeed, many standard methods rely on comparing the model performance when using the feature under study against the performance obtained when restricting its access. These works have shown improved results for methods based on modeling the dependence structure among the inputs (e.g.\ Conditional Feature Importance) rather than retraining a new model without the feature while still focusing on the relationship with the output (e.g.\ Leave-One-Covariate-Out) (see 
\cite{chamma2024statistically, paillard25a}).

The increasing use of Machine Learning (ML) has led to remarkable predictive performance on complex data (\cite{hollmann2023tabpfn, qu2026tabiclv2}). 
Consequently, \textit{model-agnostic} conditional independence approaches have become particularly relevant, as they can leverage these black-box predictors as surrogates for the underlying data-generating distribution. 

In addition, the data collection process is no longer static, and scientists increasingly require more flexible inference procedures. For instance, when a $p$-value is close to a desired significance threshold, it is tempting to continue collecting data. Similarly, data may come from multiple studies, making it unclear how to combine evidence appropriately. Standard $p$-value procedures do not generally control type-I error under such flexible settings. For this reason, there has been growing interest in the study of \textit{e-values} (\cite{grünwald2023safetesting, Ramdas_2025}). The e-value is a building block of Safe Anytime-Valid Inference (SAVI). 
Its sequential counterpart, the e-process, provides an interpretable measure of evidence (\cite{Koning2024ContinuousTU}) and allows practitioners to continuously monitor the data and decide whether additional data collection is necessary or sufficient evidence has already been obtained.

Building on \citet{shaer23a}, we aim to further study how to sequentially assign the importance of a feature within a black-box machine learning model under the model-X assumption, using conditional independence testing. Our approach compares the loss of the original model to that of a modified version where the feature's information is removed and replaced with a conditionally independent counterpart generated under the model-X assumption. Moreover, we study the statistical guarantees when this assumption is violated.


\section{Problem statement}

The model-X assumption consists of fixing a kernel $K(X \mid Z)$ for the conditional distribution of $X$ given $Z$.Thus, the model-X conditional independence null is

\begin{align}
    \mathcal{P}^\textnormal{X-CI}_{K}&:=\notag \left\{ P(X, Y, Z):\right. \\ &\left.  X\indep Y \mid Z,\ P(X\mid Z)= K(\cdot\mid Z)\right\}.\label{eq:mx-null-hyp}
\end{align}

By introducing a dummy variable $\widetilde{X}$ independent of $X$ and $Y$ given $Z$, this model-X conditional independence null can equivalently be viewed (see Proposition \ref{prop:mx_as_cond_exch}) as the conditional exchangeability of $\widetilde{X}$ and $X$ given $(Y,Z)$, where the same fixed kernel is imposed on $\widetilde{X}$:

\begin{align}\notag
    \mathcal{P}^\textnormal{X-CI}_{K}=&\left\{P(\widetilde{X}, X, Y, Z): (\widetilde{X}, X, Y, Z)\overset{\mathrm{d}}{=}(X, \widetilde{X}, Y, Z),\right. \\  &\left. \qquad P(\widetilde{X}\mid X, Y, Z)= K(\cdot\mid Z)\right\},\label{eq:mx-null-hyp2}
\end{align}

To test this null, \citet{candes2018panning} proposed the Conditional Randomization Test (CRT), which compares an importance statistic $T(X,Y,Z)$ computed on the observed data with its distribution under the null. This is achieved by replacing the original $X$ in the sample with the dummy $\widetilde{X}$ sampled from $K(X \mid Z)$. The procedure then computes the importance statistic for many conditional resamplings, namely $T(\widetilde{X},Y,Z)$, and compares them with the original statistic. This yields $p$-values, and they also extended the idea to multiple testing through the popular knockoffs. The performance of the CRT depends strongly on the choice of $T$, and a growing literature proposes different alternatives, ranging from predictive performance measures of machine learning models \citep{Tansey02012022} to coefficients of linear models \citep{liu2022fast}.



In the sequential setting, the use of model-X to test conditional independence was also studied 
by
\citet{Grünwald02042024}. They proposed the use of e-values, which are nonnegative random variables whose expectation is upper bounded by $1$ under the null. The set of model-X e-values $\mathfrak{E}$ is thus
\[
\mathfrak{E}=\left\{E\geq 0: \mathbb{E}_P[E]\leq 1 \text{ for all }P\in \mathcal{P}_{K}^\textnormal{X-CI}\right\}.
\]
To analyze the power of an e-value $E$ against a point alternative $Q$, a standard notion is the e-power \citep{vovk_wang_2024}, also known as log-power \citep{Koning2024ContinuousTU}, which is the expectation under the alternative of the logarithm of the e-value, $\mathbb{E}_Q[\log E]$. The Growth Rate Optimal (GRO) e-value \citep{grünwald2023safetesting} is the one that maximizes this criterion. \citet{Grünwald02042024} proposed the \textit{model-X} e-value
\begin{align}\label{eq:mx_definition}
    E^{\textnormal{X-CI}}
    :=
    \frac{f_{Y \mid X,Z}}{f_{Y \mid Z}}
    =
    \frac{f_{Y \mid X,Z}}{\mathbb{E}_{X \mid Z}[f_{Y \mid X,Z}]},
\end{align}
which is GRO against a point alternative $Q$ that satisfies the model-X kernel and has conditional density $f_{Y \mid Z,X}$. 
However, this density is generally unknown, and it is not straightforward to estimate it or to relate it to machine learning models, which are typically designed for prediction rather than density estimation. Thus, its practical application remains unclear. 

A closely related approach based on the loss of a predictor was proposed by \citet{shaer23a}. In particular, they introduced the \textit{eCRT}, a sequential version of the CRT. Their method relies on the fact that, under the null , for any importance statistic $T$,
\[
T = T(X,Y,Z)
\quad \text{and} \quad
\widetilde{T} = T(\widetilde{X},Y,Z)
\]
are exchangeable.
They take $T(X,Y,Z)=l(\widehat{m}(X, Z), Y)$ to be the predictive performance of a model $\widehat{m}:\mathcal{X}\times\mathcal{Z}\to \mathcal{Y}'\subseteq \mathbb{R}$, which uses inputs $(X,Z)\in\mathcal{X}\times\mathcal{Z}\subseteq \mathbb{R}^{d+1}$ to try to predict the output $Y\in\mathcal{Y}\subseteq \mathbb{R}$ under the loss function $l:\mathcal{Y}'\times \mathcal{Y}\to \mathbb{R}_+$.
They then propose the \textit{antisymmetric} e-value
\begin{align}
    1 + v\, g(T,\widetilde{T}),\label{eq:anti_def_shaer}
\end{align}
where $v$ is a scalar and $g$ is a antisymmetric, i.e., $g(b,a) = -g(a,b).$
By antisymmetry and exchangeability, it follows that under the null, $\mathbb{E}[g(T,\widetilde{T})] = 0.$
To ensure that this defines a valid e-value, one also requires that $v\,g(T,\widetilde{T}) \geq -1$, which guarantees positivity. They proposed, for instance, using the sign of  $\widetilde{T}-T$. Still, it is not clear which statistics are optimal, or how these antisymmetric e-values compare to $E^{\textnormal{X-CI}}$. 
To do so, we note that, these e-values do not directly test the initial null $\mathcal{P}^\textnormal{X-CI}_{K}$. Instead, they test 

\begin{align}\label{eq:exch_null_hyp}
\mathcal{P}^\mathrm{exch}:= \left\{P(\widetilde{T}, T): (\widetilde{T}, T)\overset{\mathrm{d}}{=}(T, \widetilde{T})\right\}.
\end{align}

This null is much larger, since any distribution in $\mathcal{P}^\textnormal{X-CI}_K$ induces an exchangeable distribution over $(T, \widetilde{T})$, but the converse does not hold.
Testing this way remains valid, but may discard substantial information. We study this further in Section \ref{sect:power}.

Finally, in practice, the assumption that $K$ is known is restrictive, and it usually needs to be estimated. Hence, we do not have exact e-values but \textit{approximate} ones, which leads to an inflation of the type-I error. Current bounds on this inflation \citep{Berrett, Grünwald02042024} are worst-case with respect to the statistic used, and are therefore too pessimistic in practice. Tighter bounds are needed to achieve meaningful control of the discoveries. In Section \ref{sect:safety}, we improve type-I error bounds when the model-X distribution is estimated, using triple-robustness results. 




\section{Analysis under the alternative}\label{sect:power}

Even though the model-X e-variable in \eqref{eq:mx_definition} is the GRO against a point alternative with density $f_{Y\mid X,Z}$, estimating this conditional distribution is challenging in practice. First, this distribution may not belong to the class of distributions over which we are searching. More generally, we do not optimize over $\mathfrak{E}$, the set of all e-variables for this null, but rather over a restricted set $\mathfrak{E}_{\textnormal{X-CI}}$. This introduces an \textit{approximation error}.

Moreover, the optimal e-variable within this restricted class is typically not available and must therefore be estimated by $\widehat{E}^{{\textnormal{X-CI}}}$, introducing an \textit{estimation error}. Neither of these errors is necessarily negligible, and both should be taken into account in the power analysis. One way to reduce these errors is to enlarge the null hypothesis or to push forward the original distribution space, allowing us to search over an easier class of e-variables, denoted by $\mathfrak{E}_{\#}$. We can then approximate the new optimal e-variable $E^{\#}$ associated with this new null hypothesis. This may reduce both the estimation and approximation errors, at the cost of introducing a \textit{null enlargement error} (see Figure \ref{fig:diagram}).

\begin{figure}
    \centering
    \hspace*{-0.2cm}
    \includegraphics[width=1.05\linewidth]{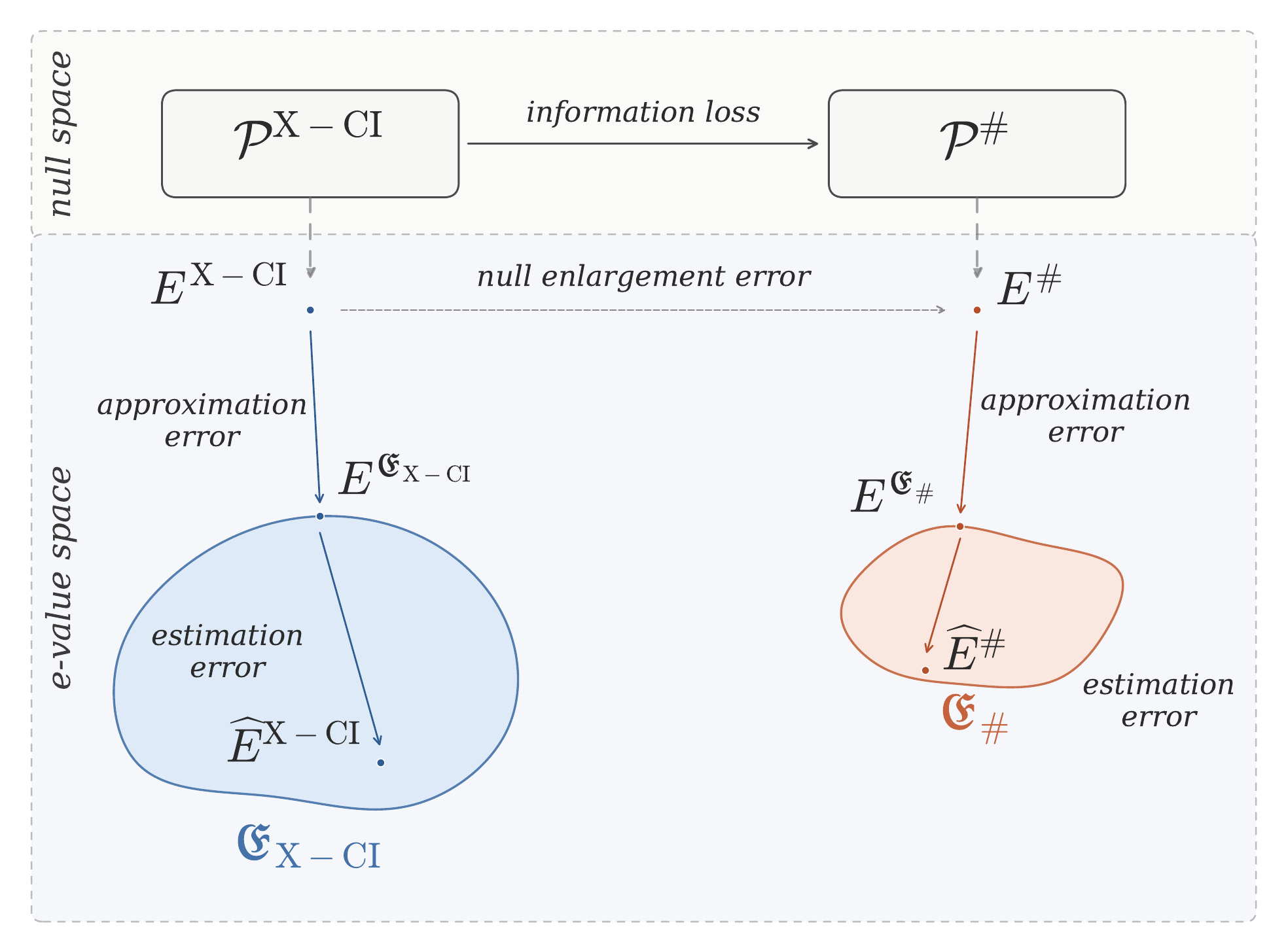}
    \caption{While distilling information induces a \textit{null enlargement error}, it can improve power by reducing the \textit{approximation} and \textit{estimation errors}. }
    \label{fig:diagram}
\end{figure}

We can decompose the e-power loss of $\widehat{E}^\#$ as
\begin{align*}
\mathbb{E}\!\left[\log E^\textnormal{X-CI}\right]
    - \mathbb{E}\!\left[\log \widehat{E}^\#\right]\qquad
    &\qquad \text{Regret of $\widehat{E}$}\\
    =\mathbb{E}\!\left[\log E^\textnormal{X-CI}\right]
    - \mathbb{E}\!\left[\log E^{\#}\right]
    &\qquad \text{Null enlargement error} \\
    {}+ \mathbb{E}\!\left[\log E^{\#}\right]
    - \mathbb{E}\!\left[\log E^{\mathfrak{E}_{\#}}\right]
    &\qquad \text{Approximation error} \\
    {}+ \mathbb{E}\!\left[\log E^{\mathfrak{E}_{\#}}\right]
    - \mathbb{E}\!\left[\log \widehat{E}^\#\right]
    &\qquad \text{Estimation error},
\end{align*}
where $E^{\mathfrak{E}_{\#}}$ denotes the optimal e-variable within the new class $\mathfrak{E}_{\#}$, which is estimated by $\widehat{E}^{\#}$.

For instance, in our setting, the joint distribution of $(X,Y,Z)$ is difficult to model. This is why we use ML models as surrogates for this distribution and then apply interpretability tools to obtain explanations from these models. Thus, we expect the push-forward not to discard too much information, and we seek to incorporate these models into the e-variables.

\subsection{Model-X vs Antisymmetric e-variables}

First, we propose an exponential class of e-variables, which provides a natural way to approximate the original GRO e-variable using ML. We then compare this class with an instance of antisymmetric e-variables, obtained by null enlargement, which relies only on the exchangeability of a push-forward through importance statistics. We prove that the latter class can outperform the GRO e-variable estimate because of the estimation and approximation errors discussed above.

To use the antisymmetric coin-betting e-value, presented below, we work with bounded losses. For the unbounded losses, see Appendix~\ref{sect:cb_unbounded}, and for the $0$-$1$ loss, where both coincide, see Appendix~\ref{sect:0-1-loss}.
The algorithms are given in Appendix~\ref{app:algo}. 
For simplicity, in the analysis we consider a fixed predictor $m$; it can directly be replaced by a sequential one.

\subsubsection{Exponential e-value}\label{sect:exp_e_val}
This proposal aligns with the \textit{model-X} density e-value \eqref{eq:mx_definition} of \citet{Grünwald02042024}  by constructing pseudo-probabilities (\cite{fast_rates}) by exponentiating the loss with a variance-like parameter $\eta$, normalized under the null:
\begin{align}
E_{\mathrm{exp}}^\eta
:=
\frac{
\exp\!\left(
-\eta\, l\!\left(m(X, Z), Y\right)
\right)
}{
\mathbb{E}_{\widetilde{X} \mid Z}\!\left[
\exp\!\left(
-\eta\, l\!\left(m(\widetilde{X}, Z), Y\right)
\right)
\right]
}.
\label{eq:exp_def}
\end{align}

\subsubsection{Coin-betting e-value}\label{sect:cb}

Coin-betting e-variables have been proven 
%
to be the only admissible e-variables
for testing the mean of a bounded random variable (\cite{CLERICO2025109550,larsson2026testinghypothesesgeneratedconstraints}). We thus study the following adaptation of coin-betting to these losses:
\begin{align}
\notag
E_{\mathrm{cb}}^\lambda
&=
1 + \lambda \left(
- l\!\left(m(X, Z), Y\right)
+ 
l\!\left(m(\widetilde{X}, Z), Y\right)
\right).
\label{eq:cb_def}
\end{align}

This is an instantiation of the \textit{antisymmetric} e-value \eqref{eq:anti_def_shaer} of \citet{shaer23a}, where there is no need to apply the sign function, since the loss is already bounded.

\begin{remark} $\eta$ and $\lambda$ may depend on $Y$ and $Z$ while still preserving the guarantees 
(see Lemma~\ref{lemm:y_z_measu}).
\end{remark}

\begin{remark}
\cite{shaer23a} observed as a practical consideration that derandomizing the dummy variable $\widetilde{X}$ when constructing antisymmetric e-variables can improve power. This is because derandomization effectively leads to a different null hypothesis, as discussed later in Section \ref{sect:null_intermediate} and Appendix \ref{appendix:hypotheses}. Nevertheless, the comparison results that follow remain unchanged when this derandomization is incorporated into the analysis.
\end{remark}

\subsubsection{Comparison of both e-values}
We first note that none of the available optimality results applies exactly in our setting. Indeed, for the exponential e-value, since the alternative density $f_{Y\mid X,Z}$ is unknown and does not necessarily admit the exponentiated-loss form considered above, it introduces a potentially large approximation and an estimation error. Likewise, the admissibility results for coin-betting e-values do not directly apply: although the loss is bounded, our setting involves the model-X null hypothesis, rather than the larger exchangeability or zero mean of the push-forward loss nulls, which introduces a null enlargement error. Consequently, there is no universally preferable option, and we show below that neither method uniformly dominates the other.


Now, let us define the \emph{predictiveness signal} of $X$ 
\[
\Delta = \e{l(m(\widetilde{X}, Z), Y)}-\e{l(m(X, Z), Y)}.
\]
If $m$ is given by a good learning algorithm, \(\Delta\) can be expected nonnegative. Indeed, under the null, we have \(\Delta=0\) by symmetry, whereas under the alternative we expect the feature \(X\), which is no longer conditionally independent of \(Y\), to provide useful predictive information. Consequently, the predictor using \(X\) should achieve a lower loss than the one using the dummy \(\widetilde X\).

This quantity can be interpreted as an effect size: the larger \(\Delta\) is, the more predictive the feature is and the easier the testing problem becomes. 
In the remainder of this section, we restrict attention to the small-signal regime. Indeed, when the signal is large, we expect the e-values to do large bets. However, the coin-betting cannot become arbitrarily large since the admissible values of \(\lambda\) are constrained to guarantee the nonnegativity. By contrast, under standard tail assumptions, the exponential e-value may grow exponentially fast in the effect size.
For instance, assume that the loss based on \(\widetilde X\) is \(\sigma\)-sub-Gaussian. 
Developing, it satisfies, 
\begin{align*} \mathbb{E}&[\mathrm{log}(E^\eta_\mathrm{exp})]\geq\eta \Delta - \frac{\sigma^2 \eta^2}{2}, \end{align*}


for any $\eta$. Optimizing over \(\eta\) yields
\[
\mathbb{E}\!\left[\log(E^{\eta^\star}_{\mathrm{exp}})\right]\geq
\frac{\Delta^2}{2\sigma^2}.
\]

Hence, the e-power of the exponential e-value grows quadratically with the signal strength \(\Delta\), and therefore diverges as the signal becomes large.

\begin{theorem}[Small-signal regime]\label{thm:small_signal}
Assume that the signal strength $\Delta$ is sufficiently small. Define
\[
\sigma_0^2 := \mathbb{V}(l(m(\widetilde {X},Z),Y)),
\qquad
\sigma_1^2 := \mathbb{V}(l(m(X,Z),Y)).
\]

Then, up to second order in $\Delta$, the optimal parameters for the exponential and coin-betting e-values satisfy

\begin{align*}
    \eta^\star
=
\frac{\Delta}{\sigma_0^2}
+
O(\Delta^2),\qquad
\lambda^\star
=
\frac{\Delta}{\sigma_1^2+\Delta^2}
+
O(\Delta^2).
\end{align*}

Moreover, the corresponding optimal  e-powers satisfy
\begin{align*}
\mathbb{E}\!\left[\log E_{\exp}^{\eta^\star}\right]
&=
\frac{\Delta^2}{2\sigma_0^2}
+
O(\Delta^3), \\
    \mathbb{E}\!\left[\log E_{\mathrm{cb}}^{\lambda^\star}\right]
&=
\frac{1}{2}
\frac{\Delta^2}{\sigma_1^2+\Delta^2}
+
O(\Delta^3).
\end{align*}
\end{theorem}

We typically expect \(\sigma_0^2\) to be larger than \(\sigma_1^2\), since the loss obtained using the informative \(X\) should typically be smaller and less variable than the one using the noisy \(\widetilde X\). Yet, there exist settings in which
$
\sigma_1^2+\Delta^2 < \sigma_0^2,
$ so that the coin-betting is more e-powerful, as illustrated in Figure~\ref{fig:exp_vs_cb}. Therefore, neither e-variable uniformly dominates the other.

Thus, even though we distill information through the push-forward introducing a null enlargement error by testing $\mathcal{P}^\mathrm{exch}$ rather than the original null $\mathcal{P}^\textnormal{X-CI}_{K}$, this can reduce the approximation error incurred by using the exponential pseudo-density and lead to higher e-power. On the other hand, if we try to estimate $f_{Y|X,Z}$ with a more complex class of learners and thus incur less approximation error, the estimation error would become the most significant source of e-power loss.

\begin{figure}[t]
    \centering
     \hspace*{-0.3cm}
    \includegraphics[width=1.05\linewidth]{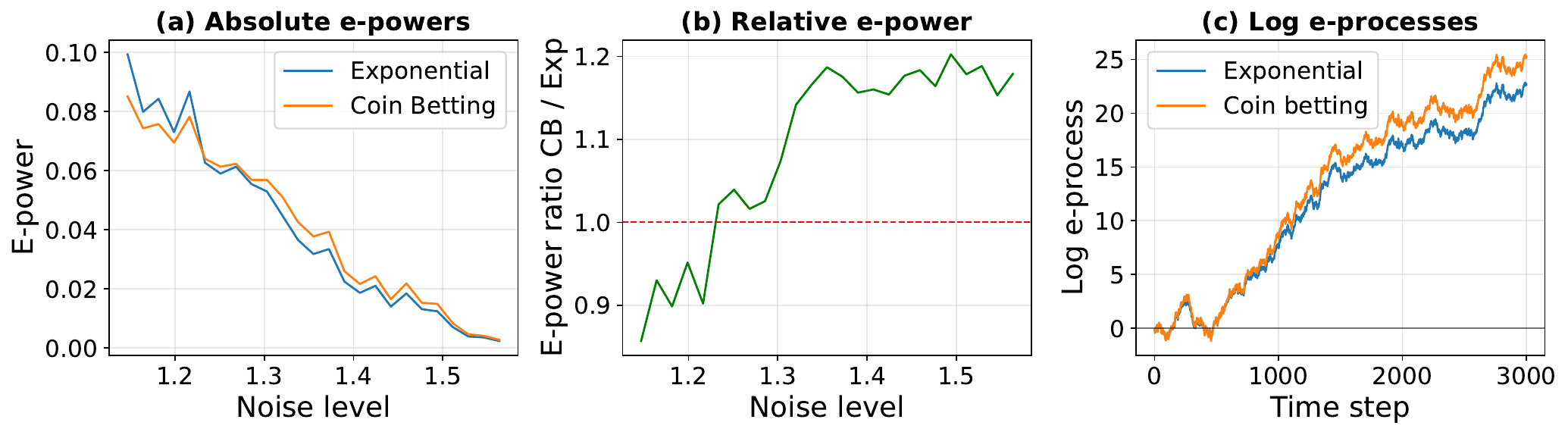}
    \caption{\textbf{Coin betting can dominate the exponential:} the left panel shows the e-powers at the optimal choices of $\eta$ and $\lambda$. As the noise level increases, the effect size shrinks and the testing problem becomes harder, and $E_{\mathrm{cb}}^{\lambda^\star}$ can outperform $E_{\mathrm{exp}}^{\eta^\star}$. The middle panel shows the ratio of their e-powers, and the right panel shows the growth of the evidence at a fixed noise level. Details are in Appendix~\ref{app:simulation}.
    }
    \label{fig:exp_vs_cb}
\end{figure}

\subsection{Optimal antisymmetric e-variable}\label{sect:kde}

Theorem~\ref{thm:small_signal} shows that, if the density does not take the form of an exponentiated loss, it is possible to construct more powerful antisymmetric e-variables \eqref{eq:anti_def_shaer}. It is then natural to ask which function $g$ is optimal.

To answer this question, we adopt a different perspective. Rather than directly testing the original model-X CI null $\mathcal{P}^\textnormal{X-CI}_{K}$ \eqref{eq:mx-null-hyp}, we first compute the importance statistics $T$ and $\widetilde{T}$ and only then, test their exchangeability, as in \cite{shaer23a}. This viewpoint changes the problem: directly characterizing $f_{Y\mid X,Z}$ is generally intractable, whereas the ML model allows us to distill the relevant conditional information into the importance statistics. Thus, we can focus on the larger exchangeability null $\mathcal{P}^\mathrm{exch}$ \eqref{eq:exch_null_hyp}.

First, note that any e-value depending only on $(T,\widetilde{T})$ is dominated by an antisymmetric e-value.

\begin{theorem}
For any e-value \(E=h(T,\widetilde{T})\) for the exchangeability null hypothesis $\mathcal{P}^\mathrm{exch}$, there exists another e-value $E'$ of the form $E'=1+g(T,\widetilde{T})$ with \(g\) antisymmetric such that $E'\geq E$
\end{theorem}
Moreover, the GRO e-variable for the exchangeability null hypothesis admits a closed-form expression.

\begin{theorem}\label{th:2_exch}
GRO e-value for the exchangeability null \eqref{eq:exch_null_hyp} against an alternative $Q$ with density \(f_{T, \widetilde{T}}\) is
\begin{align*}
    E^{\mathrm{exch}} :&= \frac{
f_{T, \widetilde{T}}(T, \widetilde{T})
}{
\frac12\bigl(f_{T, \widetilde{T}}(T, \widetilde{T})+f_{T, \widetilde{T}}(\widetilde{T},T)\bigr)}\\&=1+\frac{f_{T, \widetilde{T}}(T, \widetilde{T})-f_{T, \widetilde{T}}(\widetilde{T}, T)}{f_{T, \widetilde{T}}(T, \widetilde{T})+f_{T, \widetilde{T}}(\widetilde{T}, T)}.
\end{align*}

\end{theorem}

 In particular, the GRO is of the form $1 + g$, where $g$ is antisymmetric. The proof follows a change-of-measure similar to the symmetry result in \cite{numeraire}.



%

Although it may appear that we have replaced the problem of learning one density with that of learning another, note that the latter is only a two-dimensional marginal while the former is conditional on high-dimensional $(X, Z)$. Accordingly, it is much easier to estimate, for example using a kernel density estimator (KDE), for which standard implementations are readily available (e.g., \cite{scikit-learn}). Hence, this approach consists of sequentially estimating the joint density of the importance statistics; see Algorithm~\ref{alg:kde_e_value}.

\subsection{Intermediate null hypotheses}\label{sect:null_intermediate}

Formally studying the null hypothesis that we are actually testing not only allows us to understand when tests are safe, but also to study which e-variable is the GRO. Thus, it is of interest to analyze the null hypotheses that lie in between the model-X null \eqref{eq:mx-null-hyp}, which is our ultimate goal, and the exchangeability null \eqref{eq:exch_null_hyp}, which appears to be the feasible one. In this section, we discuss two: push-forward model-X conditional independence and the broader exchangeability that arises from derandomizing the dummy variable. For more details on other intermediate nulls and their corresponding GROs, see Appendix~\ref{appendix:hypotheses} and Table~\ref{tab:null-distillation}.


The first step is to notice that studying the raw distribution of $(\widetilde{X},X, Y, Z)$ is hard, and that we need an ML model to act as a surrogate. This is done by computing an importance statistic $T:\mathcal{X}\times \mathcal{Y}\times \mathcal{Z}\to \mathbb{R}_+$, for example the loss of a model, as done previously. Then, instead of directly requiring exchangeability of the original input and the dummy $\widetilde{X}$ conditionally on $(Y, Z)$, plus the model-X kernel, we apply the \textit{push-forward}, which can reduce the estimation and approximation error at the price of a null enlargement error:
\begin{align*}
    \mathcal{P}^\textnormal{PFX-CI}_{K^{\#}}&:=\left\{P(\widetilde{T}, T, Y, Z):(\widetilde{T}, T, Y, Z)\overset{\mathrm{d}}{=}(T, \widetilde{T}, Y, Z),\right. \\ &\left.\qquad P(\widetilde{T}\mid T, Y, Z)= K^{\#}(\cdot\mid Y, Z) \right\},
\end{align*}
where $K^{\#}$ stands for the push-forward of the initial kernel $K$ through the statistic $T(\cdot)$. It may depend on $(Y, Z)$, but $\widetilde{T}$ and $T$ are conditionally independent given $(Y, Z)$. The GRO for this null is (see Prop. \ref{prop:gro_property_proof})
\begin{align}
    E^\textnormal{PFX-CI}:=\frac{f_{T\mid Y, Z}(T\mid Y, Z)}{f_{\widetilde{T}\mid Y, Z}(T\mid Y, Z)}.
\end{align}
It is similar to the model-X GRO, replacing $(X, \widetilde{X})$ by $(T, \widetilde{T})$. It has more e-power than any e-variable using this information distillation through ML. However, like the model-X GRO, it potentially suffers from high estimation and approximation error, since it still relies on a conditional density that is hard to estimate.

Alternatively, \cite{shaer2023learning} noted that in practice, \textit{derandomizing} over the dummy variables had a great impact on the power. Indeed, the goal would be to Rao-Blackwellize the e-variable \citep{deroos2025raoblackwellizedevariables} by taking the expectation over the model-X kernel. However, since this expectation is usually unknown, they propose to take the average across $B$ resamples $\{\widetilde{T}_b\}_b^B$ with $\widetilde{T}_b:=T(\widetilde{X}_b, Y, Z)$:
\begin{align*}
    \frac{1}{B}\sum_{b=1}^B\left(1+v\,g(T, \widetilde{T}_b)\right).
\end{align*}

When doing so, the exchangeability null hypothesis \eqref{eq:exch_null_hyp} is replaced by a $B$-exchangeability null, i.e., we are no longer testing whether $T$ is exchangeable with a single resample $\widetilde{T}$, but rather with $\{\widetilde{T}_b\}_{b=1}^B$. For this $B$-exchangeable null, the GRO is 
\citep{koning2026measuringevidenceexchangeabilitygroup}
\begin{align*}
   E^{\mathrm{exch}}_B:= \frac{f_{T, \widetilde{T}^B}(T, \widetilde{T}^B)}{\frac{1}{(B+1)!}\sum_{\sigma\in \mathcal{S}_{B+1}}f_{{T, \widetilde{T}^B}}(\sigma( T, \widetilde{T}^B))},
\end{align*}
where $\widetilde{T}^B:=(\widetilde{T}_1, \ldots, \widetilde{T}_B)$, $\mathcal{S}_{B+1}$ is the set of permutations of $[B+1]$, and $\sigma(T, \widetilde{T}^B)$ is the vector $(T, \widetilde{T}^B)$ permuted by $\sigma$. Even if it has more e-power than any of the derandomized e-variables, it involves a $(B+1)$-dimensional density, which must be estimated and therefore suffers from a much higher estimation error.



\section{Analysis under the null}\label{sect:safety}

Unfortunately, the model-X assumption is rarely exactly satisfied in real-world applications, and the conditional distribution of $X$ given $Z$ needs to be estimated. 
This affects both model-X and antisymmetric e-variables. 

%
Then, the impact of this kernel estimation on the statistical guarantees still needs to be clarified. Existing bounds in the literature are often pessimistic and, in practice, appear to substantially overestimate the actual loss of type-I error control. For instance, for the CRT, \cite{Berrett} considers worst-case statistics and derives an upper bound on the type-I error inflation given by the total variation distance between the estimated and true conditional distributions. \cite{Grünwald02042024} extends this result to the sequential setting. While this bound provides a worst-case guarantee, it does not necessarily reflect the smaller loss of type-I error control observed in practice.

In this section, we account for the estimation of the conditional distribution and derive less conservative bounds on the type-I error inflation. The resulting processes from the estimated kernel are approximate e-processes, leading to approximate control.
See Appendix~\ref{sect:approx_e_values} for a formal definition of approximate e-processes and their general impact on the type-I error inflation, which we use in Theorem~\ref{thm:antisymmetric_type_I} for the approximate coin-betting e-process.
In addition, for the model-loss statistic, the error inflation can also benefit from the predictive model itself. Indeed, under conditional independence, the Bayes predictor does not depend on $X$ once $Z$ is given. Thus, as the predictive model approaches this limiting predictor, its dependence on $X$ vanishes, making the e-process progressively less sensitive to errors in the estimated conditional distribution. This yields \textit{double robustness} rates and, then, a faster vanishing of the error inflation.

We denote by $\widetilde{X}'_n$ the random variable sampled from $K_n(\cdot\mid Z_n)$, where $K_n$ is the conditional distribution from which we actually sample, representing our estimate, at time $n$, of the conditional distribution of $X_n$ given $Z_n$. We recall that $K(\cdot\mid Z_n)$ denotes the true conditional distribution of $X_n$ given $Z_n$.

\subsection{Double Robustness}

\citet{Berrett, Grünwald02042024} provide bounds that are necessarily conservative because they are obtained uniformly over the class of all statistics and do not exploit the particular structure of the one used.
However, in our setting, the statistics are
\begin{align}
    l(\widehat m_n(\widetilde X_n',Z_n),Y_n)
\quad\text{and}\quad
l(\widehat m_n(X_n,Z_n),Y_n).\label{eq:loss_dist_estim}
\end{align}

Under the conditional independence null, since $X$ is not predictive of $Y$ given $Z$, the population risk minimizer $m_0$ does not depend on $X$. 
%
Hence, the discrepancy between the distributions of \eqref{eq:loss_dist_estim} can shrink for two distinct reasons. First, the estimated conditional distribution $K_n(\cdot\mid Z)$ converges to the true $K(\cdot\mid Z).$ Second, the fitted model $\widehat m_n$ converges to the $X$-independent Bayes predictor $m_0(Z)$. Thus, even if the estimated conditional distribution is not exactly correct, the statistic becomes increasingly insensitive to discrepancies in the $X$-distribution as the fitted model approaches the null regression function.

This gives a double robustness similar to the one of Conditional Feature Importance \citep{reyerolobo2025conditionalfeatureimportancerevisited}, which is not captured by worst-case bounds based only on the distance between $K_n$ and $K$.



Let's define the excess risk at time $t$ on $(x, z, y)$
$$
r_n(x,z, y)
:=
 l(\widehat m_n(x,z), y)-l(m_0(z), y).
$$


We consider three different assumptions on the convergence of $r_n$, ranging from more to less restrictive. The first one, motivated by linear models and by \citet{lobo2026semiknockoffs}, implies continuous convergence. The second implies uniform convergence, without specifying any further regularity, and the third consists of the standard $L_2$-convergence. The more regular the assumption on this convergence, the less restrictive the resulting distributional distance. Note that the first two are conditional on $(Y, Z)$, while the last one, which is the most standard, is $L_2$-convergence specifying an explicit convergence rate.

We denote by $\mathcal W_1$ the $1$-Wasserstein distance, by $\mathrm d_{\mathrm{TV}}$ the total variation distance, and by $\mathrm d_{\chi^2}$ the square root of the chi-squared divergence.

\begin{proposition}[Double robustness]\label{prop:double_robust}
Let $K_n(\cdot\mid Z)$ be an $\mathcal F_{n-1}$-measurable estimate of $K(\cdot\mid Z)$, and suppose that, conditionally on $(\mathcal F_{n-1},Z)$, the random variables $X$ and $\widetilde X'_n$ have distributions $K(\cdot\mid Z)$ and $K_n(\cdot\mid Z)$, respectively.
Then the following hold.

\begin{enumerate}
\item[\rm (i)] \textbf{Lipschitz convergence.} Suppose that, uniformly in $z, y$, $|r_n(x,z,y)-r_n(x',z,y)| \leq b_n(z,y)\,|x-x'|$ for some $b_n$. Define $a_n(Z) := \mathcal W_1(K_n(\cdot\mid Z), K(\cdot\mid Z))$. Then, for every $z, y$,
    $\left|\mathbb E_{K_n}[r_n(\widetilde X_n,z,y)] - \mathbb E_{K}[r_n(X,z,y)]\right| \leq a_n b_n.$

\item[\rm (ii)] \textbf{Uniform convergence.} Suppose that $\sup_{x,z,y}|r_n(x,z,y)| \leq b_n(z,y)$ for some $b_n$. Define $a_n(Z) := \mathrm d_{\mathrm{TV}}(K_n(\cdot\mid Z), K(\cdot\mid Z))$. Then, for every $z, y$,
$
    \left|\mathbb E_{K_n}[r_n(\widetilde X_n,z,y)] - \mathbb E_{K}[r_n(X,z,y)]\right| \leq 2a_nb_n.
$

\item[\rm (iii)] \textbf{$L_2$ convergence.} Suppose that $\mathbb E_{K}[r_n(X,Z,Y)^2 \mid \mathcal F_{n-1}] \leq b_n^2$ for some $b_n$. Define $a_n := \sqrt{\mathbb{E}\left[\mathrm d_{\chi^2}(K_n(\cdot\mid Z) \Vert K(\cdot\mid Z))^2\right]}$. Then
$
    \left|\mathbb E[r_n(\widetilde X_n',Z,Y) - r_n(X,Z,Y)]\right| \leq a_nb_n.
$
\end{enumerate}
\end{proposition}

Consequently, in all three cases the discrepancy between the two distributions is bounded by a product $a_nb_n,$
where $a_n$ measures the error in estimating the conditional distribution and $b_n$ measures the remaining $X$-dependence of the fitted model.



\subsection{Type-I error inflation}\label{sect:trpl_rob}

For simplicity, we focus on the coin-betting e-value, i.e. on the antisymmetric function $
g(a,b)=a-b,
$
with a fixed bet fraction $\lambda$. Note, however, that under the null, this $\lambda$ also tends to zero when using standard plug-in estimates, such as the prequential plug-in estimator that we use in 
practice (also known as the growth-rate-adaptive-to-the-particular-alternative (GRAPA) procedure), or 
the Krichevsky--Trofimov-type estimator. These estimators achieve a rate of $O(n^{-1/2})$; see \cite{wang2026surenullbankruptcytestingbybetting}. Thus, in practice we 
even obtain a kind of \textit{triple robustness}. Still, to keep the treatment simple, for now
we ignore this beneficial effect on the type-I error and treat $\lambda$ as fixed.

We estimate the type-I error inflation of the practical antisymmetric e-process $E^{(n)}=\prod_{i=1}^n E_i$ by first computing the inflation of the individual e-values $E_i$ and then translating it into an inflation of the type-I error.

\begin{theorem}[Antisymmetric type-I error]
\label{thm:antisymmetric_type_I}
Assume any of the three scenarios from Proposition~\ref{prop:double_robust} hold. Let $a_n$ and $b_n$ be the corresponding distributional and model-convergence rates, and let $\delta_n$ denote the corresponding failure probability $A_n$.
Then, for $P\in\mathcal{P}_{K}^\textnormal{X-CI}$,

$$
\mathbb E_P[E_n \ind_{A_n}\mid\mathcal F_{n-1}]
\leq
1+\lambda a_nb_n.
$$

Consequently,
\begin{align*}
\mathbb P &
\left(
\sup_{1\leq n\leq N}
E^{(n)}
\geq \frac{1}{\alpha}
\right)
-\alpha
\\&\leq
\sum_{n=1}^N\delta_n
+
\alpha
\left[
\exp\left(
\lambda\sum_{n=1}^N a_nb_n
\right)-1
\right].
\end{align*}
\end{theorem}


Note that we do not require the model to exactly satisfy the assumptions underlying its convergence; it suffices for these assumptions to hold with high probability $1-\delta_n$. Also, in the first two cases, $a_nb_n$ corresponds to the expectations of those given in Proposition~\ref{prop:double_robust}.
Then, when
$
\sum_{n=1}^N a_nb_n,
$ is small enough, 
we have
\begin{align*}
\exp\left(
\sum_{n=1}^N a_nb_n
\right)-1
&=
O\left(
\sum_{n=1}^N a_nb_n
\right),
\end{align*}
and hence the type-I error inflation is controlled by
\begin{align*}
P
\left(
\sup_{1\leq n\leq N}
E^{(n)}
\geq \frac{1}{\alpha}
\right)
-\alpha
&=
O\left(
\sum_{n=1}^N\delta_n
+
\alpha\sum_{n=1}^N a_nb_n
\right),
\end{align*}
the cumulative probability of the approximation events failing and the cumulative inflation of the e-process.

In short, contrary to the previous bounds, which only rely on the convergence of the conditional distribution, our approach also captures the convergence of the model.
For example, consider the Gaussian linear setting with quadratic loss. Previous total variation bounds would only capture the $n^{-1/2}$ decay resulting from the convergence of the estimated distribution (\cite{Berrett}). In contrast, using our technique to study the inflation of the e-value, we can invoke Lemma 3.10 of \cite{reyerolobo2025conditionalfeatureimportancerevisited} to obtain a decay rate of $n^{-2}$. Finally, by incorporating the convergence rate of $\widehat{\lambda}$, we could obtain a rate of $n^{-5/2}$, substantially improving the safety guarantees.

\section{Experiments}\label{sect:exps}

We verify our power results on the HIV dataset considered by \citet{shaer23a}. We study the evidence against three features that they deemed representative of different effects on drug resistance. In Figure~\ref{fig:hiv_data}, we compare the exponential e-values with the antisymmetric e-values using different $g$-functions; see Table~\ref{tab:evalues} in the appendix for an explicit expression. See also Appendix~\ref{sect:hiv_data} for further details on the implementation and the dataset. We observe that none of the e-processes rejects for the null mutation, which shows robustness to violations of the model-X assumption. For the small effect, the exponential e-process is not able to reject, and even for the large effect it has the slowest growth rate, which shows that the approximation error is not negligible.

We have also reproduced the simulated data experiment from \citet{shaer23a} in Appendix~\ref{subsec:sim_setting}, and we have verified the \textit{triple robustness} result in Appendix~\ref{sect:robustness_safety}. The results show that the three estimates indeed converge to the true quantities, so that the e-process remains almost constant under the null and does not reject. Finally, in Appendix~\ref{appendix:effects_of_derandomization} we further discuss and simulate the effects of derandomization.



All code is available at 
\url{https://github.com/Antares01/sequential_ml_testing}.


\begin{figure}
    \centering
    \includegraphics[width=1\linewidth]{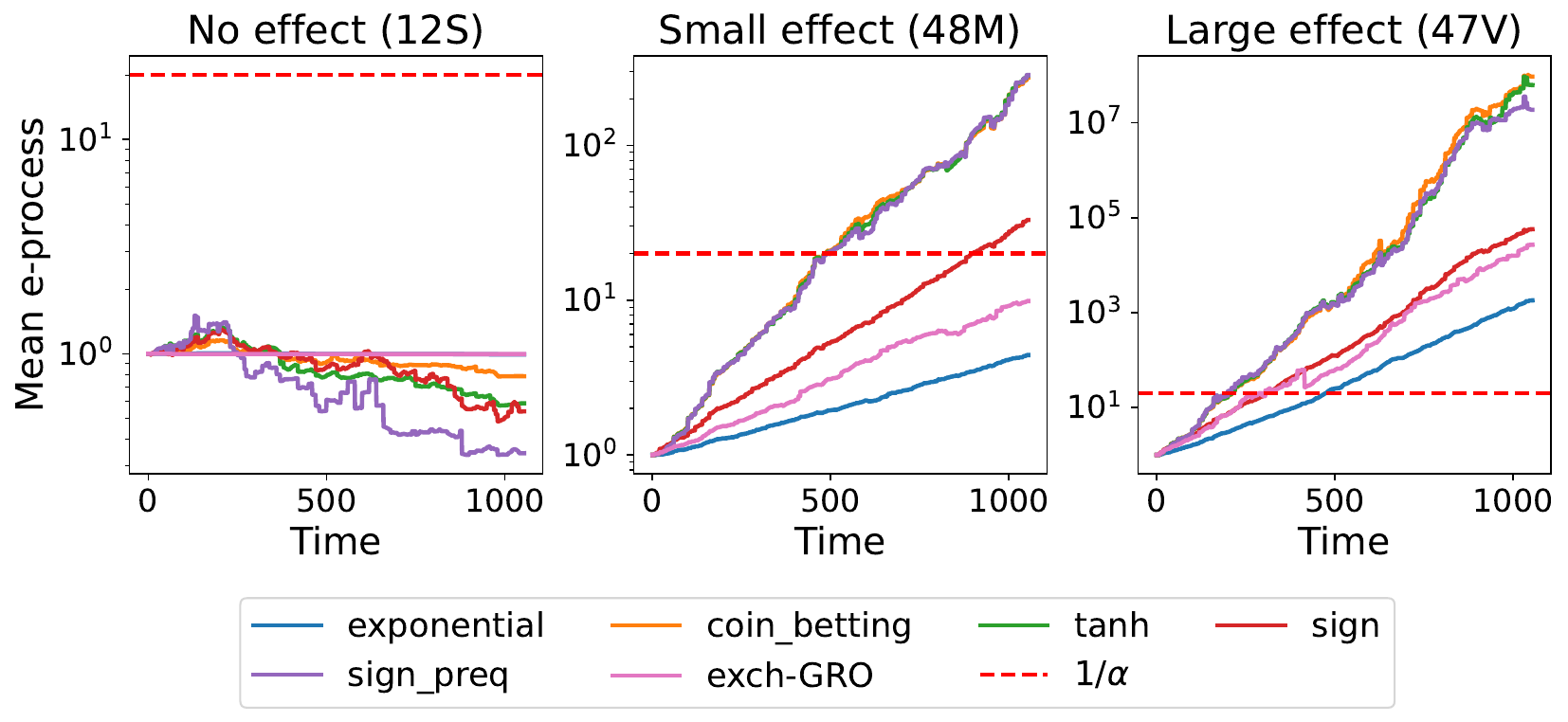}
    \caption{\textbf{HIV dataset:} Evidence collected against the null (effect on drug resistance) for three different HIV mutations. On the left, for an unimportant mutation, none of the e-processes reject. In the center, for a small signal, most antisymmetric e-proceses reject the null and outperform the exponential, which does not reject. On the right, in a large signal regime all the e-processes reject the null.}
    \label{fig:hiv_data}
\end{figure}
%

\section{Discussion}

In this work we studied how to accommodate ML models in sequential testing through e-values. We first proposed a natural way to estimate the GRO using the exponential e-value. Due to density misspecification, this induces an approximation error, and there is also an estimation error within this exponential class. Thus, it is not uniformly better than the antisymmetric e-value class, which is adapted to the exchangeability null, a much larger null hypothesis, and thereby introduces a null enlargement error. This error decomposition accounts for the fact that estimating the point alternative is not straightforward. This analysis is essential not only to better understand the behavior of both classes in practice and to compare them, but also to formally explore intermediate null hypotheses, so as to generate e-values with a lower null enlargement error by taking better advantage of the inner structure of the model-X conditional independence null.


This has great applicability because in the sequential world it is possible to successively select a mixture of tests using the previous data points. The approximation/null enlargement weighting can be data-driven and dynamic over time. 
For instance, while using coin-betting from an initial point may be a good idea, since only $\lambda$ needs to be estimated and thus the estimation error may be low, once more data is acquired, a more accurate estimate of the density e-value is obtained, yielding larger power than coin-betting, and one is then able to iteratively shift more weight toward the latter in the mixture, reducing the null enlargement error. 
Practically studying how to effectively instantiate this mixture, taking into account the intermediate null hypotheses, GRO e-values, and model training, is important future work.

Finally, challenging the model-X assumption is very relevant in practice. In the proposed new analysis, we have first been able to use the tool of approximate e-values to study the inflation of the e-value, and we have then also taken into account the structure of the statistic to exploit double robustness properties, thereby providing a more reliable tool for sequential testing using ML.

\subsection*{AI use statement}

In this work, we used generative AI tools to assist in the writing of proofs, implement methods and assist with translation.
We have not used generative AI tools to generate synthetic data sets, help develop theoretical models or conceptual frameworks, formulate mathematical claims, provide critical ingredients for proving mathematical claims, propose or refine hypotheses and interpret results. Design or provide feedback on research  methodology or experiments, clean and reformat dataset and support qualitative and thematic data analysis are not applicable to this work.
We have reviewed all AI-assisted work: LLM-generated code was verified and tested for correctness by the authors and all LLM-assisted proofs were checked line by line by the authors.
We take responsibility for the final content of this work, including text, claims, or artifacts produced with the aid of generative AI.



\subsubsection*{Acknowledgements}
This project has received funding from the European Union's Horizon Europe research and innovation programme under grant agreement No 101120237 (ELIAS); from state aid managed by the Agence Nationale de la Recherche under grant ANR-23-CE23-0016 and under grant ANR-18-EURE-0023 EUR MINT.
It is partially funded by the European Union via ERC Advanced Project 101142168 FLEX. Views and opinions expressed are however those of the author(s) only and do not necessarily reflect those of the European Union or the European Research Council Executive Agency. Neither the European Union nor the granting authority can be held responsible for them

\bibliographystyle{plainnat}   
\bibliography{biblio_AISTATS} 

\section*{Checklist}



\begin{enumerate}

  \item For all models and algorithms presented, check if you include:
  \begin{enumerate}
    \item A clear description of the mathematical setting, assumptions, algorithm, and/or model. [Yes]
    \item An analysis of the properties and complexity (time, space, sample size) of any algorithm. [Yes]
    \item (Optional) Anonymized source code, with specification of all dependencies, including external libraries. [Yes]
  \end{enumerate}

  \item For any theoretical claim, check if you include:
  \begin{enumerate}
    \item Statements of the full set of assumptions of all theoretical results. [Yes]
    \item Complete proofs of all theoretical results. [Yes]
    \item Clear explanations of any assumptions. [Yes]
  \end{enumerate}

  \item For all figures and tables that present empirical results, check if you include:
  \begin{enumerate}
    \item The code, data, and instructions needed to reproduce the main experimental results (either in the supplemental material or as a URL). [Yes]
    \item All the training details (e.g., data splits, hyperparameters, how they were chosen). [Yes]
    \item A clear definition of the specific measure or statistics and error bars (e.g., with respect to the random seed after running experiments multiple times). [Yes]
    \item A description of the computing infrastructure used. (e.g., type of GPUs, internal cluster, or cloud provider). [Not Applicable]
  \end{enumerate}

  \item If you are using existing assets (e.g., code, data, models) or curating/releasing new assets, check if you include:
  \begin{enumerate}
    \item Citations of the creator If your work uses existing assets. [Yes]
    \item The license information of the assets, if applicable. [Not Applicable]
    \item New assets either in the supplemental material or as a URL, if applicable. [Not Applicable]
    \item Information about consent from data providers/curators. [Not Applicable]
    \item Discussion of sensible content if applicable, e.g., personally identifiable information or offensive content. [Not Applicable]
  \end{enumerate}

  \item If you used crowdsourcing or conducted research with human subjects, check if you include:
  \begin{enumerate}
    \item The full text of instructions given to participants and screenshots. [Not Applicable]
    \item Descriptions of potential participant risks, with links to Institutional Review Board (IRB) approvals if applicable. [Not Applicable]
    \item The estimated hourly wage paid to participants and the total amount spent on participant compensation. [Not Applicable]
  \end{enumerate}

\end{enumerate}

\clearpage
\appendix
\thispagestyle{empty}

\onecolumn
\aistatstitle{
Supplementary Materials}

\section{Algorithms}\label{app:algo}

In this section we present the general algorithms, which are agnostic to how the model is retrained: they accommodate \textit{any} machine learning model, which can be updated either completely from scratch or online for speed. They work for any classification or regression task for which a loss measuring performance can be provided. Similarly, the 2-Exchangeability GRO e-value can accommodate any kernel density estimate. Thus, the procedures are flexible enough for many different settings.

\subsection{Exponential e-process}
\begin{algorithm}
\caption{Sequential Exponential Estimation Procedure\label{alg:ml_e_value}}

\textbf{Input:} Initial estimator $(\widehat{m}_1,\widehat{\eta}_1)$ from pretraining

Set $E_0 \gets 1$

\medskip

\textbf{for} $i = 1,2,\dots,\tau$ \textbf{do}
\begin{enumerate}
    \item Sample $(X_i,Z_i,Y_i)$.
    
    \item Compute
    \[
    E_{\mathrm{exp},i}
    :=
    \frac{
    \exp\!\left(
    -\widehat{\eta}_i(Y_i,Z_i)\,
    l\!\left(\widehat{m}_i(X_i,Z_i),Y_i\right)
    \right)
    }{
    \mathbb{E}_{\widetilde{X}\mid Z_i}\!\left[
    \exp\!\left(
    -\widehat{\eta}_i(Y_i,Z_i)\,
    l\!\left(\widehat{m}_i(\widetilde{X},Z_i),Y_i\right)
    \right)
    \right]
    }.
    \]
    
    \item Update $(\widehat{m}_{i+1},\widehat{\eta}_{i+1})$
    using the data
    \[
    [i] = \{(X_j,Z_j,Y_j)\}_{j=1}^{i}.
    \]
\end{enumerate}

\textbf{end for}

\medskip

\textbf{Return}
\[
\prod_{i=1}^{\tau} E_{\mathrm{exp},i}.
\]

\end{algorithm}

The stopping time $\tau$ is left unspecified because the guarantees allow for optional stopping, and therefore for choosing a different stopping time depending on the particular experiment, for example if the researcher obtains more funding or if it is suddenly forced to stop (\cite{grünwald2023safetesting}). However, a standard choice consists of stopping as soon as the process reaches the desired threshold of $1/\alpha$.

Note that, in practice, one simple --- although computationally expensive --- way to update the model is to retrain it from scratch using the extended training dataset containing the newly observed point. However, this can be prohibitively intensive, and lighter alternatives are available, such as updating the existing model through stochastic gradient descent (SGD).

Regarding the choice of the optimal $\eta$, although it may in principle depend on $(Z,Y)$ (see Lemma \ref{lemm:y_z_measu}) and therefore be $\sigma(Z,Y)$-measurable, allowing for such flexibility considerably increases the complexity of the learning problem. In practice, we therefore typically restrict ourselves to constant values of $\eta$. For instance, $\eta$ can be learned in a prequential fashion by maximizing the empirical past log-growth:
\begin{align}
\widehat{\eta}_i
=
\underset{\eta \in \mathbb{R}}{\mathrm{argmax}}
\sum_{k=1}^{i}
\log
\left(
\frac{
e^{-\eta\, l(\widehat{m}_k(X_k,Z_k),Y_k)}
}{
\mathbb{E}_{\widetilde{X}\mid Z_k}
\left[
e^{-\eta\, l(\widehat{m}_k(\widetilde{X},Z_k),Y_k)}
\right]
}
\right).\label{eq:eta_optim}
\end{align}

In practice this can be easily done with a grid of possible $\eta$s. Otherwise, it is possible to do a mixture instead of this plug-in estimate.



\subsection{Standard coin-betting and coin-betting with unbounded losses}\label{sect:cb_unbounded}

Coin-betting e-value algorithm consists of replacing Step 2 from Algorithm \ref{alg:ml_e_value} with
\begin{align*}
    E_{\mathrm{cb},i}:= 1+\widehat{\lambda}_i (\mathbb{E}_{\widetilde{X}\mid Z_i}[l(\widehat{m}_i(\widetilde{X}, Z_i), Y_i)]-l(\widehat{m}_i(X_i, Z_i), Y_i)),
\end{align*}
where $\widehat{\lambda}_i$ can also be learnt prequentially via

\begin{align}
\widehat{\lambda}_i
=
\underset{\lambda \in [0, 1/M]}{\mathrm{argmax}}
\sum_{k=1}^{i}
\log
\left(
1+\lambda (\mathbb{E}_{\widetilde{X}\mid Z_k}[l(\widehat{m}_k(\widetilde{X}, Z_k), Y_k)]-l(\widehat{m}_k(X_k, Z_k), Y_k))
\right),\label{eq:lambda_optim}
\end{align}

assuming that the loss is bounded by $M$.

However, even if the loss is bounded, its bounds may be unknown \emph{a priori}. Therefore, we can introduce a constant $M$ to control the loss:
\begin{align}
E_{\mathrm{cb}}^{\lambda, M} 
= 1 + \lambda \frac{1}{M}\left(
-\, l\!\left(\widehat{m}(X, Z), Y\right)\wedge M
+ \mathbb{E}_{\widetilde{X}\mid Z}\!\left[l\!\left(\widehat{m}(\widetilde{X}, Z), Y\right)\wedge M\right]
\right).
\end{align}
By doing so, we can apply this coin-betting e-variable even when the loss is not inherently bounded.

Moreover, the parameter $M$, similarly to $\lambda$, can be learned in a prequential manner from past data. Specifically,
\[
(\hat{\lambda}_n, \hat{M}_n)
:= \underset{(\lambda, M)\in [0,1]\times \mathbb{R}_+}{\mathrm{argmax}}
\prod_{i=1}^{n-1}
\left(
1 + \lambda \frac{1}{M}\left(
-\, l\!\left(\widehat{m}_i(X_i, Z_i), Y_i\right)\wedge M
+ \mathbb{E}_{\widetilde{X}_i\mid Z_i}\!\left[l\!\left(\widehat{m}_i(\widetilde{X}_i, Z_i), Y_i\right)\wedge M\right]
\right)
\right).
\]

Alternatively, since the main requirement is to ensure that the e-value remains nonnegative, we can directly clip the difference both from above and below. This yields
\begin{align}
E_{\mathrm{cb}} 
= 1 + \lambda \frac{1}{M}
\left(
\left(
-\, l\!\left(\widehat{m}(X, Z), Y\right)
+ l\!\left(\widehat{m}(\widetilde{X}, Z), Y\right)
\right)
\wedge M
\right)\vee (-M).\label{eq:cb_cap}
\end{align}

Note that this bounded coin-betting is indeed a \textit{generalization of the sign e-value} considered in \cite{shaer23a}, given by $1+\lambda\,\mathrm{sign}\!\left(l(\widehat{m}(\widetilde{X}, Z), Y)-l(\widehat{m}(X, Z), Y)\right)$ (see also Section~\ref{sect:e_values_anti_exps}), since by taking $M$ to $0$, we observe that the magnitude is completely lost and only the sign of the difference is preserved.

Finally, we can apply Rao--Blackwellization to remove the randomness due to $\widetilde{X}$ by taking conditional expectations. 

\subsection{2-Exchangeability GRO}

\begin{algorithm}[H]
\caption{Sequential 2-Exchangeability GRO Procedure}\label{alg:kde_e_value}

\textbf{Input:} Initial estimator $(\widehat{m}_1,\widehat{\eta}_1, \hat{f}_{T,\widetilde{T}, 1})$

Set $E_0 \gets 1$

\medskip

\textbf{for} $i = 1,2,\dots,\tau$ \textbf{do}
\begin{enumerate}
    \item Sample $(X_i,Z_i,Y_i)$.

    \item Compute $t_{i}= l\!\left(\widehat{m}_i(X_i,Z_i),Y_i\right)$ and $\widetilde{t}_{ i} = l\!\left(\widehat{m}_i(\widetilde{X},Z_i),Y_i\right)$
    \item Compute
    $$E_{i}^\mathrm{exch}:=\mathbb{E}_{\widetilde{X}\mid Z}\left[\frac{
\hat{f}_{T, \widetilde{T}, i}(t_i, \widetilde{t}_i)
}{
\frac12\bigl(\hat{f}_{T, \widetilde{T}, i}(t_{ i}, \widetilde{t}_{ i})+\hat{f}_{T, \widetilde{T}, i}(\widetilde{t}_{ i}, t_{i})\bigr)}\right]$$
    
    \item Update $(\widehat{m}_{i+1})$
    using the data $ [i] = \{(X_j,Z_j,Y_j)\}_{j=1}^{i}$
    and update $\hat{f}_{T, \widetilde{T}, i+1}$ using $\{(t_{ j}, \widetilde{t}_{ j})\}_{j=1}^{i}$.
\end{enumerate}

\textbf{end for}

\medskip

\textbf{Return}
\[
\prod_{i=1}^{\tau} E_{i}^\mathrm{exch}.
\]

\end{algorithm}

In Algorithm \ref{alg:kde_e_value} we provide a general implementation without specifying the kernel estimate. Note that there are several possible ways to implement this e-process, and exploring these alternatives is an interesting direction for future work. In particular, the ML model is updated sequentially, so the distribution of the importance statistics may evolve over time. Consequently, it may be advantageous for the density estimator to gradually \textit{forget} older observations and place greater emphasis on more recent ones. This can be achieved, for example, by using a sliding-window approach or by assigning exponentially decaying (or otherwise adaptive) weights to past observations, thereby allowing the estimated density to better track changes in the distribution of the importance statistics. Note also that it is optimized without taking the derandomization into account, and therefore loses a lot of power compared to the other antisymmetric e-values (see Appendix~\ref{appendix:effects_of_derandomization} and Table~\ref{tab:derandomization}). Thus, one idea would be to select a parametric model for the kernel density estimation and to optimize its parameters prequentially, like the parameter $\lambda$ of the coin-betting, rather than nonparametrically with the previous losses.

In addition, similarly to the antisymmetric e-variables, we have also derandomized the e-value by averaging over the dummy variable $\widetilde{X}$.

\paragraph{Derandomization in practice.} In all these algorithms, even if the kernel is known, the conditional expectation usually needs to be estimated. To do so, we can simply average over many resamplings from the conditional distribution. In the exponential e-variable (Algorithm~\ref{alg:ml_e_value}), we also need to include the original $X$ in the average in order to retain type-I error control, as done in \citet{Grünwald02042024}.

\paragraph{Conditional density estimation.} If the kernel is unknown and needs to be estimated, as discussed in Section~\ref{sect:safety}, we can also do so sequentially using the previous data. Thus, in the step where the model and the parameters of the e-value are updated, we would add an update of $K_{i+1}$ using $\{(X_l, Z_l)\}_{l=1}^i$. In Appendix~\ref{sect:cond_dist} we further discuss how we have implemented this conditional estimation in practice.


\section{Intermediate Null Hypotheses}
\label{appendix:hypotheses}

The aim of this section is to build a conceptual bridge between the model-X CI null \citep{Grünwald02042024} and a larger null hypothesis that is implicitly being tested by e-variables from \citet{shaer23a}. It is then a detailed and extended version of Section \ref{sect:null_intermediate}. We do so by providing a detailed analysis of several null hypotheses and corresponding GRO e-variables that lie ``in between", as shown in Table \ref{tab:null-distillation}.\\ 

\begin{table}[H]
\centering
\small
\renewcommand{\arraystretch}{1.45}
\caption{
Sequence of null hypotheses obtained by progressively distilling
information. Moving downward corresponds to retaining less information
and hence considering a larger null hypothesis. The GRO e-variable associated with
each null is shown in the final column (see \ref{prop:gro_property_proof}).
}
\label{tab:null-distillation}

\begin{tabular}{p{4cm} p{7cm} p{4cm}}
\toprule \textbf{Symbol and Name}
& \textbf{Hypothesis definition}
& \textbf{GRO e-variable}
\\
\midrule

$\mathcal{P}^\textnormal{X-CI}_{K}$: Model-X Conditional Independence
&
$\left\{P(\widetilde{X}, X, Y, Z): (\widetilde{X}, X, Y, Z)\overset{\mathrm{d}}{=}(X, \widetilde{X}, Y, Z), \right.$\newline
\hspace*{2em}$\left. P(\widetilde{X}\mid X, Y, Z)= K(\cdot\mid Z)\right\}$
&
$
\frac{f_{Y\mid X,Z}(y|x,z)}{f_{Y\mid Z}(y|z)}
$
\\[6mm]

& $\Big\downarrow$ \;\;\;\; \textcolor{red}{Coarsening error} & 
\\[5mm]

$\mathcal{P}^\textnormal{PFX-CI}_{K^{\#}}$: Push-Forward Model-X Conditional Independence 
&
$\left\{P(\widetilde{T}, T, Y, Z):(\widetilde{T}, T, Y, Z)\overset{\mathrm{d}}{=}(T, \widetilde{T}, Y, Z), \right.$\newline
\hspace*{2em}$\left.P(\widetilde{T}\mid T, Y, Z)= K^{\#}(\cdot\mid Y, Z) \right\}$
&

$\frac{f_{T\mid Y, Z}(t\mid y, z)}{f_{\widetilde{T}\mid Y, Z}(t \mid y, z)}$

\\[6mm]

& $\Big\downarrow$ \;\;\;\; \textcolor{red}{Randomization error} & 
\\[5mm]

$\mathcal{P}^{\text{cond exch}}_B$: $B$-Conditional exchangeability 
&
$\Big\{P(\widetilde{T}^B,T,Y,Z): (\widetilde{T}^B,T,Y,Z)\overset{\mathrm{d}}{=}\sigma (\widetilde{T}^B,T,Y,Z),$\newline
\hspace*{2em}$ \forall \sigma \in \mathcal{S}_{B+1} \Big\}$
&
$
\frac{f_{T,\widetilde{T}^B|Y,Z}(t,\widetilde{t}^B|y,z)}{\frac{1}{(B+1)!}\sum_{\sigma}f_{T,\widetilde{T}^B|Y,Z}(\sigma(t,\widetilde{t}^B)|y,z)}
$
\\[6mm]

& $\Big\downarrow$ \;\;\;\; \textcolor{red}{Coarsening error} & 
\\[5mm]

$\mathcal{P}^{\text{exch}}_B$: $B$-Exchangeability 
&
$\left\{P(\widetilde{T}^B,T): (\widetilde{T}^B,T)\overset{\mathrm{d}}{=}\sigma (\widetilde{T}^B,T), \forall \sigma \in \mathcal{S}_{B+1} \right\}$ 
&
$
\frac{f_{T,\widetilde{T}^B}(t,\widetilde{t}^B)}{\frac{1}{(B+1)!}\sum_{\sigma}f_{T,\widetilde{T}^B}(\sigma(t,\widetilde{t}^B))}
$
\\[4mm]

& $\Big\downarrow$ \;\;\;\; \textcolor{red}{Randomization error} & 
\\[5mm]

$\mathcal{P}^\mathrm{exch}$: Exchangeability
&
$\left\{P(\widetilde{T}, T): (\widetilde{T}, T)\overset{\mathrm{d}}{=}(T, \widetilde{T})\right\}$
&
$
\frac{
f_{T,\widetilde{T}}(t,\widetilde t)
}{
\frac{1}{2}f_{T,\widetilde{T}}(t,\widetilde t)
+
\frac{1}{2}f_{T,\widetilde{T}}(\widetilde t,t)
}
$
\\[4mm]

& $\Big\downarrow$ \;\;\;\; \textcolor{red}{Coarsening error} & 
\\[5mm]

$\mathcal{P}^\mathrm{sym}$: Symmetry 
&
$\left\{P(\widetilde{T}, T): \widetilde{T}- T \overset{\mathrm{d}}{=} T- \widetilde{T} \right\}$
&
$
\frac{
f_{\widetilde{T}-T}(\widetilde{t}-t)
}{
\frac{1}{2}f_{\widetilde{T}-T}(\widetilde{t}-t)
+
\frac{1}{2}f_{\widetilde{T}-T}(t-\widetilde{t})
}
$
\\

\bottomrule
\end{tabular}

\end{table}

\paragraph{The X-CI null.} The model-X conditional independence null hypothesis for a single datapoint was originally defined as:
\begin{equation*}
    \mathcal{P}^{\text{X-CI}}_{K} = \{P(X,Y,Z): X \indep Y | Z, P(X|Z) \sim K_Z\}
\end{equation*}

An alternative representation of the same null can be obtained by making use of an additional, dummy variable $\widetilde{X}$ conditionally sampled from $K_Z$ and independently of $(X,Y)$.

\begin{proposition} 
\label{prop:mx_as_cond_exch}
$\mathcal{P}^{\text{X-CI}}_{K}$ is equivalent to the hypothesis
\begin{equation}
\label{eq:mx_as_cond_exch}
     \{P(\widetilde{X},X,Y,Z): (\widetilde{X},X,Y,Z) \overset{d}{=} (X, \widetilde{X},Y,Z),  P(\widetilde{X}|X,Y,Z) \sim K_Z \}
\end{equation}
in the sense that any distribution $P(X,Y,Z) \in \mathcal{P}^{\text{X-CI}}_{K}$ can be extended to an element of \eqref{eq:mx_as_cond_exch} by adding the additional variable $\widetilde{X}$ distributed as $K_Z$ given $(X,Y,Z)$, and conversely, any distribution $P(\widetilde{X},X,Y,Z)$ in \eqref{eq:mx_as_cond_exch} has an $(X,Y,Z)$-marginal that belongs to $\mathcal{P}^{\text{X-CI}}_{K}$.
\end{proposition}

 This, seemingly unnecessary, extension allows us to represent the conditional independence testing problem as one of testing \textbf{conditional exchangeability}, making the comparison with the setting of \citet{shaer23a} more natural. From now on, we make no distinction between $\mathcal{P}^{\text{X-CI}}_{K}$ and \eqref{eq:mx_as_cond_exch}.

Note that $\widetilde{X}$ is indeed a dummy variable. The same $(X,Y,Z)$ - conditional distribution will always be assumed under the alternative hypothesis, meaning that the statistic $(X,Y,Z)$ is sufficient for both null and alternative hypotheses, and thus $\widetilde{X}$ will never appear as an argument in an optimal test (in the Neyman-Pearson or the e-power sense). 

The GRO e-variable for testing the MX null against any point alternative $Q$ with $Q(X|Z) = K_Z$ can now be represented in several equivalent ways. Assume that $f_{\widetilde{X},X,Y,Z}$ is the probability density function of the joint distribution of the data under the alternative:

\begin{equation*}
    E^{\text{X-CI}}(x,y,z) = \frac{f_{Y|X,Z}(y|x,z)}{f_{Y|Z}(y|z)}= \frac{f_{X|Y,Z}(x|y,z)}{f_{X|Z}(x|z)}= \frac{f_{X|Y,Z}(x|y,z)}{f_{\widetilde{X}|Y,Z}(x|y,z)}
\end{equation*}
are equivalent ways to represent the GRO e-variable.

What all these representations have in common is that in order to use the statistic for testing in practice, one of the conditional densities $f_{Y|X,Z}$ or $f_{X|Y,Z}$ has to be estimated, which can be hard. To alleviate this task, one could focus on an importance statistic  $T$, i.e. a function of the data triplet $(X,Y,Z)$ whose conditional distribution could be easier to estimate than that of $X$ given $Y,Z$. 
For example, fix $T = l(\hat{m}(X,Z),Y), \widetilde{T} = l(\hat{m}(\widetilde{X},Z),Y)$. 

\paragraph{The PFX-CI null.} Thus, based on the importance statistic, we can formulate a new, push-forward model-X conditional independence null hypothesis:

\begin{equation*}
    \mathcal{P}^{\text{PFX-CI}}_{K^{\#}} = \{P(\widetilde{T},T,Y,Z): (\widetilde{T},T,Y,Z) \overset{d}{=} (T, \widetilde{T},Y,Z),  P(\widetilde{T}|T,Y,Z) \sim K^{\#}_{Y,Z} \},
\end{equation*}

where $K^{\#}_{y,z}$ is the push-forward distribution $\widetilde{T}|Y=y,Z=z$.
The growth-rate optimal e-variable for this null is given by the conditional likelihood ratio:
\begin{equation*}
    E ^{\text{PFX-CI}}(t,y,z) =  \frac{f_{T|Y,Z}(t|y,z)}{f_{\widetilde{T}|Y,Z}(t|y,z)},
\end{equation*}
similar the third representation of $E^{\text{X-CI}}$. 

The e-power of $E^{\text{PFX-CI}}$ is upper bounded by that of $E^{\text{X-CI}}$, which is given by the conditional mutual information $I(X;Y\mid Z)$ between $X$ and $Y$ given $Z$:
\begin{equation*}
    \mathbb{E}_{Y,Z} [D_{KL}(Q(T|Y,Z) ||Q(\widetilde{T}|Y,Z))] \leq \mathbb{E}_{Y,Z} [D_{KL}(Q(X|Y,Z) ||Q(\widetilde{X}|Y,Z))] \overset{def}{=} I_Q(X;Y|Z).
\end{equation*}
This power gap represents the \textbf{coarsening error}, and is a consequence of the data-processing inequality:
\begin{equation*}
    D_{KL}(Q(T|Y=y,Z=z) ||Q(\widetilde{T}|Y=y,Z=z)) \leq D_{KL}(Q(X|Y=y,Z=z) ||Q(\widetilde{X}|Y=y,Z=z)) \text{ } \forall y,z.
\end{equation*}

In a practical setting, an e-power increase could come from adequate importance statistics choices, which could make learning the true conditional density $f_{T|Y,Z}$ easier than learning  $f_{X|Y,Z}$ or $f_{Y|X,Z}$ from the data, thus lowering the approximation and estimation errors.\\

\paragraph{The $B$-conditional exchangeability nulls.}  As we have seen in previous work \citep{shaer23a}, the model-X kernel $K$ can be used to generate a finite number $B$ of $\widetilde{X}$ samples, making antisymmetric e-variables practically applicable and yielding a monotonic e-power increase as $B$ grows (see Proposition \ref{prop:rb_monotonically_increases_power}). Any e-variable of that form is necessarily weaker than the limiting -- Rao-Blackwellized e-variable -- if such an object is well defined, which makes use of infinitely many derandomization samples ($B = \infty$). The following family of null hypotheses in our hierarchy and the corresponding GRO e-variables allow for a simple demonstration of this fact. The $B$ - conditional exchangeability null is defined as:

\begin{equation*}
    \mathcal{P}^{\text{cond exch}}_B = \{P(\widetilde{T}^B,T,Y,Z): (\widetilde{T}^B,T,Y,Z) \overset{d}{=} (\sigma(\widetilde{T}^B,T),Y,Z) \text{ for any permutation } \sigma\in \mathcal{S}_{B+1} \}.
\end{equation*}

Let the point alternative distribution $Q(\widetilde{T}^B,T,Y,Z)$  again be given by the density function $f$. The GRO e-variable for testing the conditional exchangeability null with $B$ derandomization samples is (Proposition \ref{prop:gro_property_proof}) 
\begin{equation*}
   E^{\text{cond exch}}_B(t,\widetilde{t}^B,y,z) = \frac{f_{T, \widetilde{T}^B|Y,Z}(t,\widetilde{t}^B|y,z)}{\frac{1}{(B+1)!}\sum_{\sigma} f_{T,\widetilde{T}^B|Y,Z}(\sigma(t,\widetilde{t}^B)|y,z)},
\end{equation*}
which, assuming conditional independence of $\{\widetilde{T}_1,...,\widetilde{T}_B,T\}$ given $(Y,Z)$ under the alternative, and identical conditional distributions of $\{\widetilde{T}_i \}$ converges to (see Appendix \ref{sect:limiting_gro_cond_exch})
\begin{equation*}
    \frac{f_{T|Y,Z}(T|Y,Z)}{f_{\widetilde{T}|Y,Z}(T|Y,Z)},
\end{equation*}
$Q$-a.s.
Therefore, we recover the GRO e-variable for the push-forward null hypothesis as a limit of GRO e-variables for $B$ - conditional exchangeability nulls. In the conditional exchangeability null the model-X assumption completely disappears, and is replaced with $B$ samples from the conditional distribution $K_{Y,Z}^{\#}$. Thus, going from the PFX-CI null to the conditional exchangeability null can be viewed as a way of simplifying the testing problem by removing perfect model-X knowledge and replacing it with (many) samples from the kernel. The e-power price that is paid for making this simplification is the \textbf{randomization error}:
\begin{equation*}
    \mathbb{E}[\ln(E^{\text{PFX-CI}})] - \mathbb{E}[\ln(E^{\text{cond exch}}_B)],
\end{equation*}
which is decreasing in $B$ and vanishes as $B \to \infty$ (see Proposition \ref{prop:cond_exch_e_power_seq}).\\

\paragraph{The $B$-exchangeability nulls.} The $B$-exchangeability family of null hypotheses can be viewed as the nulls that were implicitly tested by \citet{shaer23a} with their antisymmetric e-variables, whenever they employed derandomization:

\begin{equation*}
    \mathcal{P}_B^{\textnormal{exch}} = \{P(\widetilde{T}^B,T): (\widetilde{T}^B,T) = \sigma(\widetilde{T}^B,T) \text{ for any permutation } \sigma\in \mathcal{S}_{B+1} \}.
\end{equation*}
The difference compared to the conditional exchangeability null is that variables $(Y,Z)$ are treated as latent. Only the importance statistics $T, \widetilde{T}^B$ are used by the test-statistics, and exchangeability is tested only marginally, i.e. on average over $(Y,Z)$.

The GRO e-variables for this family are given by
\begin{equation*}
    E_B^{\text{exch}}(t,\widetilde{t}^B) = \frac{f_{T,\widetilde{T}^B}(t,\widetilde{t}^B)}{\frac{1}{(B+1)!} \sum_{\sigma} f_{T,\widetilde{T}^B}(\sigma(t,\widetilde{t}^B))}.
\end{equation*}
$(Y,Z)$-conditional exchangeability implies marginal exchangeability, meaning that 
\begin{equation*}
    \mathcal{P}_B^{\text{cond exch}} \subset \mathcal{P}_B^{\text{exch}}
\end{equation*}
holds if we only focus on the $(T, \widetilde{T}^B)$ - marginals. 
Therefore, $E_B^{\text{exch}}$ can only have a lower e-power than $E_B^{\text{cond exch}}$.

Using $\varphi(y,z) = Q(\widetilde{T}|Y=y,Z=z)$, evaluated at $(T,\widetilde{T}^B)$, the e-variable $Q$ - a.s. converges to
\begin{equation*}
    \frac{f_{T|\varphi(Y,Z)}(T|\varphi(Y,Z))}{  f_{\widetilde{T}|Y,Z}(T|Y,Z)} = \frac{f_{T|\varphi(Y,Z)}(T|\varphi(Y,Z))}{  f_{\widetilde{T}|\varphi(Y,Z)}(T|\varphi(Y,Z))}
\end{equation*}
 as $B \to \infty$  under regularity conditions (see Proposition \ref{prop:gro_for_exch_limit}).\\
 The $Q$-a.s. equality of the denominators holds because by definition of $\varphi$
 \begin{equation*}
    \widetilde{T} \indep (Y,Z)| \varphi(Y,Z),  \text{ i.e. } Q(\widetilde{T}|Y,Z) = Q(\widetilde{T}|\varphi(Y,Z)),
 \end{equation*}
so there exists a version of the conditional density $f_{\widetilde{T}|Y,Z}$ which only depends on $(y,z)$ through $\varphi(y,z)$.

A direct consequence of this result together with the convexity property of the KL-divergence is that even with infinite derandomization samples, the GRO e-variables (and thus all e-variables) for the exchangeability nulls generally can not achieve the e-power of $E^{\text{PFX-CI}}$:
\begin{equation*}
    \mathbb{E}_{\varphi(Y,Z)}[D_{KL}(f_{T|\varphi(Y,Z)} ||f_{\widetilde{T}|\varphi(Y,Z)})] \leq \mathbb{E}_{Y,Z}[D_{KL}(f_{T|Y,Z} ||f_{\widetilde{T}|Y,Z})].
\end{equation*}

The inequality can be strict when $\varphi$ is not injective, i.e. when the true underlying $(Y,Z)$ pair is not identifiable based on an infinite sample of $\widetilde{T}_i$.\\
Furthermore, the analysis suggests another approach for model-X-based conditional independence testing. Instead of estimating the conditional density $f_{T|Y,Z}$, which can lead to large estimation errors, pick a function $s$, estimate the simpler conditional density $f_{T|s(Y,Z)}$ and test with
\begin{equation*}
    \frac{f_{T|s(Y,Z)}(t|s(y,z))}{f_{\widetilde{T}|Y,Z}(t|y,z)}.
\end{equation*}

As before, the conditional density in the denominator is known exactly, or can be estimated to arbitrary precision by sampling many $\widetilde{X}$ values from $K_Z$ and computing $\widetilde{T}$ samples. As long as the model-X assumption holds, this conditional likelihood ratio always provides type-I error guarantees, no matter what we choose for $s$ and $f_{T|s(Y,Z)}$. A coarser function $s$ leads to easier conditional density estimation in the numerator, but also to a larger loss of e-power under perfect estimation. Furthermore, we know that for a fine enough $s$, i.e. any $s$ such that $\varphi$ is a function of $s$, the achieved e-power with perfect conditional density estimation is greater than the e-power of any exchangeability e-variable, no matter how many derandomization samples are used to enhance it.\\

\paragraph{The exchangeability and symmetry nulls.} In their theoretical analysis, \citet{shaer23a}, similarly to knockoffs~\citep{candes2018panning}, focus on testing the exchangeability of $T$ and a single sample $\widetilde{T}$, without leveraging derandomization. Thus, in reality, they test the exchangeability null hypothesis
\begin{equation*}
    \mathcal{P}^{\text{exch}} = \{P(T,\widetilde{T}): (T,\widetilde{T}) \overset{d}{=} (\widetilde{T},T)\},
\end{equation*}
proposing e-variables of the form
\begin{equation*}
    1 + \lambda g(\widetilde{T},T),
\end{equation*}
where $g$ is a 2-dimensional antisymmetric function as in \eqref{eq:anti_def_shaer}.
All practical antisymmetric e-variables they provide have the even more restricted form
\begin{equation*}
    1 + \lambda g(\widetilde{T}-T),
\end{equation*}
where $g$ is now a 1-dimensional function antisymmetric about $0$ ($g(-x)=-g(x)$), and as such they are e-variables for the larger -- symmetry null hypothesis

\begin{equation*}
    \mathcal{P}^{\text{sym}} = \{P(T, \widetilde{T}): \widetilde{T}-T \overset{d}{=} T - \widetilde{T}\}.
\end{equation*}

The corresponding GRO e-variables for testing against an alternative $Q(T, \widetilde{T})$ with density $f$ are 
\begin{equation*}
    E^{\text{exch}}(t,\widetilde{t}) = \frac{f_{T,\widetilde{T}}(t,\widetilde{t})}{\frac{1}{2}f_{T,\widetilde{T}}(t,\widetilde{t}) + \frac{1}{2}f_{T,\widetilde{T}}(\widetilde{t},t)}, \text{  } E^{\text{sym}}(\widetilde{t}-t) = \frac{f_{\widetilde{T}-T}(\widetilde{t}-t)}{\frac{1}{2}f_{\widetilde{T}-T}(\widetilde{t}-t) + \frac{1}{2}f_{\widetilde{T}-T}(t-\widetilde{t})}.
\end{equation*}

The practical implementation and usage of these GRO e-variables is relatively easily achievable. They only require the estimation of a respectively 2 or 1 - dimensional density function, instead of a $(B+1)$-dimensional density, or the conditional densities of $Y|X,Z$ or $T|Y,Z$, which as we have seen is necessary for optimally testing the nulls that are higher in our hierarchy.\\
However, the optimality of these GRO e-variables is theoretically guaranteed \textbf{only if no derandomization is performed}. If the e-variables are enhanced by derandomization, the antisymmetric e-variables of $1+\lambda g$ form can \textbf{gain more e-power} than both $E^{\text{exch}}$ and $E^{\text{sym}}$ gain, and thus can surpass their e-powers. In other words, the \textbf{GRO property is not preserved after derandomization}, which shows the crucial importance of including the number of derandomization samples in the definition of our intermediate nulls.
Theory ensures:
\begin{align*}
    \mathbb{E}\left[\ln\left(1+\lambda g(T, \widetilde{T})\right)\right] &\leq \mathbb{E}\left[\ln\left(E^{\text{exch}}(T,\widetilde{T})\right)\right], \text{ and }\\
    \mathbb{E}\left[\ln\left( 1 + \frac{1}{B} \sum_{i=1}^B \lambda g(T, \widetilde{T_i})\right)\right] &\leq \mathbb{E}\left[\ln\left(E^{\text{exch}}_B(T, \widetilde{T}^B)\right)\right],
\end{align*}
for any antisymmetric $g$, and under any point alternative $Q$ with density $f$ (on which the e-variables on the right-hand side depend), but in general
\begin{equation*}
    \mathbb{E}\left[\ln\left( 1 + \frac{1}{B} \sum_{i=1}^B \lambda g(T, \widetilde{T_i})\right)\right] \not \leq \mathbb{E}\left[\ln\left(\frac{1}{B} \sum_{i=1}^B E^{\text{exch}}(T, \widetilde{T_i})\right)\right].
\end{equation*}
On the other hand, $E^{\text{exch}}$ and $E^{\text{sym}}$ make use of asymmetry signals beyond non-0 mean and median values of $\widetilde{T}-T$, and could achieve a large e-power even in settings when: $1+ \lambda \operatorname{sign}(\widetilde{T}-T), 1+\lambda \tanh(\widetilde{T}-T)$ and $1+\frac{\lambda}{M} \operatorname{clip}(\widetilde{T}-T, -M,M)$ all have e-power very close or equal to $0$ at the optimal betting parameters.\\
For more details on the effects of derandomization on the e-powers of antisymmetric e-variables, $E^{\text{exch}}$ and $E^{\text{sym}}$ see Appendix \ref{appendix:effects_of_derandomization}.

\section{From approximate e-variables to approximate type-I error}\label{sect:approx_e_values}

We begin by formalizing the notion of an approximate e-variable and an approximate e-process.

\begin{definition}[Approximate e-variable]
Let $\mathcal{P}$ be a class of probability measures. A nonnegative random variable $E$ is an $(\epsilon,\delta)$-approximate e-variable for $\mathcal{P}$ if, for every $P\in\mathcal{P}$, there exists an event $A$ such that
\begin{align*}
\mathbb{E}_{P}\left[E\ind_{A}\right]
\leq 1+\epsilon
\qquad
\text{ and }
\qquad
P(A)
\geq 1-\delta.
\end{align*}
\end{definition}

Note that this definition is not identical to the definition of an approximate e-variable given in \cite{ignatiadis2025asymptoticcompoundevaluesmultiple}, although it implies the latter; see, for example, Proposition 10.2 of \cite{Ramdas_2025}. For simplicity, we use fixed values of $\epsilon$ and $\delta$ rather than allowing them to depend on $P$.

The definition extends naturally to an approximate e-process.

\begin{definition}[Approximate e-process]\label{def:approx_e_process}
Let $(E_t)_{t\geq 1}$ be a nonnegative process adapted to $(\mathcal{F}_t)_{t\geq 0}$, with $E_0=1$. We say that $(E_n)_{n\geq 1}$ is an ${(\epsilon_n,\delta_n)}_{t\geq1}$-approximate e-process for $\mathcal{P}$ if, for every $P\in\mathcal{P}$ and every $t\geq1$, there exists an event $A_n\in\mathcal{F}_n$ such that
\begin{align*}
\mathbb{E}_{P}
\left[
E_t\ind_{A_n}
\middle|
\mathcal{F}_{n-1}
\right]
&\leq
(1+\epsilon_n)E_{n-1},
\end{align*}
and
\begin{align*}
P(A_n\mid\mathcal{F}_{n-1})
\geq 1- \delta_n
\qquad\text{a.s.}
\end{align*}
\end{definition}

For this approximate e-processes we can obtain an approximate type-I error control.

\begin{proposition}\label{prop:approx_e_type_I}
Let $(E_t)_{t\geq1}$ be an ${(\epsilon_n,\delta_n)}_{t\geq1}$-approximate e-process for $\mathcal{P}$, with $E_0=1$. 
Then, for every $P\in\mathcal{P}$ and every $N\geq1$,
\begin{align*}
P
\left(
\sup_{1\leq n\leq N}
E_{n}
\geq \frac{1}{\alpha}
\right)
&\leq
\sum_{n=1}^N \delta_n
+
\alpha
\prod_{t=1}^N(1+\epsilon_n)
\\ &\leq
\sum_{n=1}^N \delta_n
+
\alpha
\exp\left(
\sum_{n=1}^N\epsilon_n
\right).
\end{align*}
\end{proposition}

The proposition shows that the type-I error incurred by using an approximate e-process is controlled by two terms: the probability that at least one of the approximation events fails, and the cumulative multiplicative inflation of the e-process. In particular,
\begin{align*}
P
\left(
\sup_{1\leq n\leq N}
E^{(n)}
\geq \frac{1}{\alpha}
\right)
-\alpha
&\leq
\sum_{n=1}^N\delta_n
+
\alpha
\left[
\exp\left(
\sum_{n=1}^N\epsilon_n
\right)-1
\right].
\end{align*}

Consequently, whenever
$
\sum_{n=1}^N\epsilon_n=o(1),
$
we have
\begin{align*}
\exp\left(
\sum_{n=1}^N\epsilon_n
\right)-1
&=
O\left(
\sum_{n=1}^N\epsilon_n
\right),
\end{align*}
and hence
\begin{align*}
P
\left(
\sup_{1\leq n\leq N}
E^{(n)}
\geq \frac{1}{\alpha}
\right)
-\alpha
&=
O\left(
\sum_{n=1}^N\delta_n
+
\alpha\sum_{n=1}^N\epsilon_n
\right).
\end{align*}
Thus, the type-I error inflation resulting from using an approximate rather than an exact e-process is controlled by the cumulative probability of the approximation events failing and the cumulative multiplicative inflation of the e-process.

\section{Proofs}

\subsection{$\sigma(Y, Z)$-measurability of $\lambda$ and $\eta$}

\begin{lemma}
Any $\hat{\eta}_n$ that is $\sigma\!\left(
Y_n, Z_n, \{(X_i, Y_i, Z_i)\}_{i=1}^{n-1}
\right)\text{-measurable}$ provides a valid exponential supermartingale. Similarly, any $\hat{\lambda}_n$ that is
$\sigma\!\left(
Y_n, Z_n, \{(X_i, Y_i, Z_i)\}_{i=1}^{n-1}
\right)\text{-measurable}$ provides a valid coin-betting supermartingale.\label{lemm:y_z_measu}
\end{lemma}
\begin{proof}
We denote by $\mathcal{D}_{n-1}:=\{(X_i, Y_i, Z_i)\}_{i=1}^{n-1}$ the previous observed data.
We observe that
\begin{align*}
    \mathbb{E}\left[\frac{\exp\!\left(-\widehat{\eta}\, l\!\left(\widehat{m}(X, Z), Y\right)\right)}
{\mathbb{E}_{\widetilde{X} \mid Z}\!\left[\exp\!\left(-\widehat{\eta}\, l\!\left(\widehat{m}(\widetilde{X}, Z), Y\right)\right)\right]}\right]&=\mathbb{E}\left[\mathbb{E}\left[\frac{\exp\!\left(-\widehat{\eta}\, l\!\left(\widehat{m}(X, Z), Y\right)\right)}
{\mathbb{E}_{\widetilde{X} \mid Z}\!\left[\exp\!\left(-\widehat{\eta}\, l\!\left(\widehat{m}(\widetilde{X}, Z), Y\right)\right)\right]}\mid Y, Z, \mathcal{D}_{n-1}\right]\right]\\&=\mathbb{E}\left[\mathbb{E}_{\widetilde{X} \mid Z}\left[\frac{\exp\!\left(-\widehat{\eta}\, l\!\left(\widehat{m}(X, Z), Y\right)\right)}
{\mathbb{E}_{\widetilde{X} \mid Z}\!\left[\exp\!\left(-\widehat{\eta}\, l\!\left(\widehat{m}(\widetilde{X}, Z), Y\right)\right)\right]}\right]\right]=1.
\end{align*}\label{rmk:depend_y_z}
The term $\mathbb{E}_{\widetilde{X} \mid Z}\left[\frac{\exp\!\left(-\widehat{\eta}\, l\!\left(\widehat{m}(X, Z), Y\right)\right)}
{\mathbb{E}_{\widetilde{X} \mid Z}\!\left[\exp\!\left(-\widehat{\eta}\, l\!\left(\widehat{m}(\widetilde{X}, Z), Y\right)\right)\right]}\right]$ simplifies to 1 for every choice of $\widehat{\eta}(Y, Z)$ dependent on $Y, Z, \mathcal{D}_{n-1}$.
\end{proof}
Therefore, applying Ville's inequality, we obtain valid-$\alpha$-type-I error control for the sequential threshold at $1/\alpha$.

\subsection{Optimal e-values for the $0$-$1$ loss}\label{sect:0-1-loss}
We consider the special case of the \(0\)-\(1\) loss. 
Surprisingly, in this setting, the e-values coincide when evaluated at their respective optimal $(Y, Z)$-dependent choices of \(\lambda\) and \(\eta\). 
Define
\[
p_1
:=
P_{X\mid Y,Z}
\bigl(
m(X,Z)\neq Y
\mid Y,Z
\bigr),
\]
\[
p_0
:=
P_{\widetilde X\mid Z}
\bigl(
m(\widetilde X,Z)\neq Y
\mid Y,Z
\bigr).
\]

These correspond to the conditional probabilities of incorrect prediction when using \(X\) and \(\widetilde X\), respectively.

\begin{theorem}[$\ell_{0-1}$ loss]
Assume that the loss is given by the $0$--$1$ loss. Then, the optimal value of $\lambda$ maximizing the coin-betting e-value is
\[
\lambda_\star
=
\frac{p_1-p_0}{p_0(1-p_0)}.
\]

Similarly, the optimal $\eta$ for the exponential e-value \eqref{eq:exp_def} is
\[
\eta_\star
=
\log\!\left(
\frac{p_0(1-p_1)}{p_1(1-p_0)}
\right).
\]

With these choices, both e-values coincide and are given by
\[
E^{\eta_\star}_{\mathrm{exp}}
=
E^{\lambda_\star}_{\mathrm{cb}}
=
\begin{cases}
\dfrac{1-p_1}{1-p_0},
& m(X,Z)=Y, \\[1.2ex]
\dfrac{p_1}{p_0},
& m(X,Z)\neq Y.
\end{cases}
\]
\end{theorem}

\begin{proof}
In this $0$--$1$ loss setting we can provide an explicit expression for the e-values
that optimize the e-power across the coin-betting and model-X classes, which
turn out to coincide.

We begin by optimizing $\eta$ for $E^\eta_{\mathrm{ML}}$. Observe that
\begin{align*}
\mathbb{E}_{X \mid Y,Z}\left[\log\left(E^\eta_{\mathrm{ML}}\right)\right]
&=
\mathbb{E}_{X \mid Y,Z}\left[
-\eta \ind_{\{m(X,Z)\neq Y\}}
-
\log
\mathbb{E}_{\widetilde{X}\mid Z}\left[
e^{-\eta \ind_{\{m(\widetilde X,Z)\neq Y\}}}
\right]
\right].
\end{align*}

Differentiating with respect to $\eta$ gives
\begin{align*}
-
\mathbb{E}_{X\mid Y,Z}\!\left[\ind_{\{m(X,Z)\neq Y\}}\right]
+
\frac{
\mathbb{E}_{\widetilde{X}\mid Z}\left[
\ind_{\{m(\widetilde X,Z)\neq Y\}}
e^{-\eta \ind_{\{m(\widetilde X,Z)\neq Y\}}}
\right]
}{
\mathbb{E}_{\widetilde{X}\mid Z}\left[
e^{-\eta \ind_{\{m(\widetilde X,Z)\neq Y\}}}
\right]
}.
\end{align*}

Let
\[
p_1 = P_{X\mid Y,Z}(m(X,Z)\neq Y \mid Y,Z),
\qquad
p_0 = P_{\widetilde X\mid Z}(m(\widetilde X,Z)\neq Y \mid Y,Z).
\]

Note that
\[
\mathbb{E}_{\widetilde{X}\mid Z}\left[
\ind_{\{m(\widetilde X,Z)\neq Y\}}
e^{-\eta \ind_{\{m(\widetilde X,Z)\neq Y\}}}
\right]
=
e^{-\eta} p_0.
\]

Thus the first-order condition becomes
\begin{align*}
-p_1
+
\frac{e^{-\eta_\star} p_0}{e^{-\eta_\star} p_0 + 1 - p_0}
=0.
\end{align*}

Solving yields
\begin{align*}
\eta_\star
=
\log
\frac{p_0(1-p_1)}{p_1(1-p_0)}.
\end{align*}

With this optimal $\eta_\star$, the e-variable $E^{\eta_\star}_{\mathrm{ML}}$
is
\begin{align*}
E^{\eta_\star}_{\mathrm{ML}}
&=
\frac{e^{-\eta_\star l(m(X,Z),Y)}}
{\mathbb{E}_{\widetilde{X}\mid Z}\left[e^{-\eta_\star l(m(\widetilde X,Z),Y)}\right]}
\\
&=
\frac{\left(e^{-\eta_\star}\right)^{\ind_{\{m(X,Z)\neq Y\}}}}
{p_0 e^{-\eta_\star} + 1 - p_0}
\\
&=
\frac{\left(\frac{p_1(1-p_0)}{p_0(1-p_1)}\right)^{\ind_{\{m(X,Z)\neq Y\}}}}
{p_0\frac{p_1(1-p_0)}{p_0(1-p_1)} + 1 - p_0}
\\
&=
\left(
\frac{p_1(1-p_0)}{p_0(1-p_1)}
\right)^{\ind_{\{m(X,Z)\neq Y\}}}
\frac{1-p_1}{1-p_0}.
\end{align*}

Therefore, there are two cases:
\[
E =
\begin{cases}
\dfrac{1-p_1}{1-p_0}, & m(X,Z)=Y, \\
\dfrac{p_1}{p_0}, & m(X,Z)\neq Y.
\end{cases}
\]

\medskip

We can proceed similarly for the coin-betting e-value. Optimizing
$\lambda$ yields
\[
\lambda_\star
=
\frac{p_1-p_0}{p_0(p_0-1)}.
\]

Substituting this value into the e-variable gives two cases.

If the prediction is correct, $Y=m(X,Z)$, then
\[
E_{\mathrm{CB}}
=
1+\lambda_\star p_0
=
1+\frac{p_1-p_0}{p_0-1}
=
\frac{1-p_1}{1-p_0}.
\]

If the prediction is incorrect,
\[
E_{\mathrm{CB}}
=
1+\frac{p_1-p_0}{p_0(p_0-1)}(p_0-1)
=
\frac{p_1}{p_0}.
\]

Thus the coin-betting and model-X optimal e-variables coincide.

\end{proof}

\subsection{E-power in the small-signal regime}

\paragraph{Exponential e-power expansion:} 
\begin{proof}
First, note that the e-power of the exponential e-value, denoted by $\mathcal{P}_{\exp}^{\eta}$, is given by
\[
\mathcal{P}_{\exp}^{\eta}
=
\mathbb{E}\!\left[\log\!\left(E_{\exp}^{\eta}\right)\right]
=
\mathbb{E}\!\left[
-\eta\, l(m(X,Z),Y)
-
\log
\mathbb{E}_{\widetilde X \mid Z}
\!\left[
e^{-\eta l(m(\widetilde X,Z),Y)}
\right]
\right].
\]

Recall that the inner expectation is formally taken with respect to
\(P_{\widetilde X\mid X,Y,Z}\). However, since
\(\widetilde X \indep (X,Y)\mid Z\), we simply write
\(P_{\widetilde X\mid Z}\) to avoid cumbersome notation.

Define
\[
\psi(\eta)
:=
\log
\mathbb{E}_{\widetilde X \mid Z}
\!\left[
e^{-\eta l(m(\widetilde X,Z),Y)}
\right].
\]

Using a Taylor expansion around \(\eta=0\),
\[
\psi(\eta)
=
\psi(0)
+
\eta \psi'(0)
+
\frac{\eta^{2}}{2}\psi''(0)
+
O(\eta^{3}).
\]

Since
\[
\psi(0)
=
\log
\mathbb{E}_{\widetilde X\mid Z}[1]
=
0,
\]
we have
\[
\psi(0)=0.
\]

Differentiating,
\[
\psi'(\eta)
=
-\frac{
\mathbb{E}_{\widetilde X \mid Z}
\!\left[
l(m(\widetilde X,Z),Y)
e^{-\eta l(m(\widetilde X,Z),Y)}
\right]
}{
\mathbb{E}_{\widetilde X \mid Z}
\!\left[
e^{-\eta l(m(\widetilde X,Z),Y)}
\right]
},
\]
and therefore
\[
\psi'(0)
=
-\mathbb{E}_{\widetilde X\mid Z}
\!\left[
l(m(\widetilde X,Z),Y)
\right].
\]

Differentiating once more yields
\begin{align*}
\psi''(\eta)
&=
\frac{
\mathbb{E}_{\widetilde X\mid Z}
\!\left[
l(m(\widetilde X,Z),Y)^2
e^{-\eta l(m(\widetilde X,Z),Y)}
\right]
\,
\mathbb{E}_{\widetilde X\mid Z}
\!\left[
e^{-\eta l(m(\widetilde X,Z),Y)}
\right]
}{
\left(
\mathbb{E}_{\widetilde X\mid Z}
\!\left[
e^{-\eta l(m(\widetilde X,Z),Y)}
\right]
\right)^2
}
\\
&\qquad
-
\frac{
\left(
\mathbb{E}_{\widetilde X\mid Z}
\!\left[
l(m(\widetilde X,Z),Y)
e^{-\eta l(m(\widetilde X,Z),Y)}
\right]
\right)^2
}{
\left(
\mathbb{E}_{\widetilde X\mid Z}
\!\left[
e^{-\eta l(m(\widetilde X,Z),Y)}
\right]
\right)^2
}.
\end{align*}

Evaluating at \(\eta=0\), we obtain
\begin{align*}
\psi''(0)
&=
\mathbb{E}_{\widetilde X\mid Z}
\!\left[
l(m(\widetilde X,Z),Y)^2
\right]
-
\left(
\mathbb{E}_{\widetilde X\mid Z}
\!\left[
l(m(\widetilde X,Z),Y)
\right]
\right)^2
\\
&=
\operatorname{Var}_{\widetilde X\mid Z}
\!\left(
l(m(\widetilde X,Z),Y)
\right)
\\
&=:
\sigma_0^2.
\end{align*}

Substituting the Taylor expansion into the e-power expression gives
\begin{align*}
\mathcal P_{\exp}^{\eta}
&=
-\eta\,
\mathbb E\!\left[l(m(X,Z),Y)\right]
\\
&\qquad
-
\mathbb E\!\left[
-\eta\,
\mathbb E_{\widetilde X\mid Z}
\!\left[
l(m(\widetilde X,Z),Y)
\right]
+
\frac{\eta^2}{2}\,
\operatorname{Var}_{\widetilde X\mid Z}
\!\left(
l(m(\widetilde X,Z),Y)
\right)
+
O(\eta^3)
\right]
\\
&=
-\eta
\Bigl(
\mathbb E[l(m(X,Z),Y)]
-
\mathbb E\bigl[
\mathbb E_{\widetilde X\mid Z}
[l(m(\widetilde X,Z),Y)]
\bigr]
\Bigr)
\\
&\qquad
-
\frac{\eta^2}{2}\,
\mathbb E
\!\left[
\operatorname{Var}_{\widetilde X\mid Z}
\!\left(
l(m(\widetilde X,Z),Y)
\right)
\right]
+
O(\eta^3).
\end{align*}

Recall that
\[
\Delta
:=
\mathbb E\bigl[l(m(\widetilde X,Z),Y)
\bigr]
-
\mathbb E[l(m(X,Z),Y)]
\]
and
\[
\sigma_0^2
:=
\mathbb E
\!\left[
\operatorname{Var}_{\widetilde X\mid Z}
\!\left(
l(m(\widetilde X,Z),Y)
\right)
\right].
\]

Then
\[
\mathcal P_{\exp}^{\eta}
=
\eta \Delta
-
\frac{\eta^2}{2}\sigma_0^2
+
O(\eta^3).
\]

Ignoring higher-order terms and maximizing the quadratic approximation yields
\[
\eta^*
=
\frac{\Delta}{\sigma_0^2},
\]
and
\[
\mathcal P_{\exp}^*
=
\frac{\Delta^2}{2\sigma_0^2}.
\]

\end{proof}

\paragraph{Coin-betting e-power expansion:} 
\begin{proof}
We recall that the coin-betting e-value is given by
\[
E_{\mathrm{cb}}^\lambda
=
1
-
\lambda\!\left(
l\!\left(\widehat{m}(X,Z),Y\right)
-
\mathbb{E}_{\widetilde X \mid Z}\!\left[
l\!\left(\widehat{m}(\widetilde X,Z),Y\right)
\right]
\right),
\]
where the expectation is taken only over $\widetilde{X}$ condition both on $Z$ and $Y$, but for simplicity, and to underline that $\widetilde{X}$ is independent of $Y$, we hide this conditioning on $Y$.
Its e-power is
\[
\mathcal{P}_{\mathrm{cb}}^\lambda
=
\mathbb{E}\!\left[\log\!\left(E_{\mathrm{cb}}^\lambda\right)\right]
=
\mathbb{E}\!\left[
\log\!\left(
1
-
\lambda\!\left(
l\!\left(\widehat{m}(X,Z),Y\right)
-
\mathbb{E}_{\widetilde X \mid Z}\!\left[
l\!\left(\widehat{m}(\widetilde X,Z),Y\right)
\right]
\right)
\right)
\right].
\]

Using the Taylor expansion
\[
\log(1+x)
=
x-\frac{x^2}{2}+O(x^3),
\]
we obtain
\[
\mathcal{P}_{\mathrm{cb}}^\lambda
=
\mathbb{E}\!\left[
\lambda D
+\frac{1}{2}\lambda^2 D^2
\right]
+
O(\lambda^3),
\]
where
\[
D
:=
\mathbb{E}_{\widetilde X \mid Z}\!\left[
l\!\left(\widehat{m}(\widetilde X,Z),Y\right)
\right]-
l\!\left(\widehat{m}(X,Z),Y\right).
\]

Recall that the predictiveness signal is
\[
\Delta
:=
\mathbb{E}[D]
=
\mathbb{E}\!\left[
\mathbb{E}_{\widetilde X \mid Z}\!\left[
l\!\left(\widehat{m}(\widetilde X,Z),Y\right)
\right]
\right]
-
\mathbb{E}\!\left[
l\!\left(\widehat{m}(X,Z),Y\right)
\right],
\]
and define
\[
\sigma_1'^2
:=
\mathbb{E}[D^2].
\]

Then
\[
\mathcal{P}_{\mathrm{cb}}^\lambda
=
\lambda \Delta
-
\frac{1}{2}\lambda^2 \sigma_1'^2
+
O(\lambda^3).
\]

Ignoring higher-order terms, the optimal value of $\lambda$ is obtained by maximizing the quadratic approximation, yielding
\[
\lambda^*
=
\frac{\Delta}{\sigma_1'^2}.
\]

Substituting this value back gives
\[
\mathcal{P}_{\mathrm{cb}}^*
=
\frac{\Delta^2}{2\sigma_1'^2}.
\]

We conclude by noting that $\sigma_1'^2=\sigma_1^2+\Delta^2$.
\end{proof}

\subsection{Antisymmetric e-value's class}

\subsubsection{Optimality of the 1 + Antisymmetric for exchangeability}

\begin{proof}
Indeed, for any e-variable \(E=h(t,\widetilde t)\), we have
\begin{align*}
h(t,\widetilde t)
&=
\frac12\bigl(h(t,\widetilde t)+h(\widetilde t,t)\bigr)
+
\frac12\bigl(h(t,\widetilde t)-h(\widetilde t,t)\bigr).
\end{align*}

The second term is antisymmetric. Moreover, for all $t$ and $\widetilde t$ the following is in the null hypothesis,
\[
P=\frac12\bigl(\delta_{t,\widetilde t}+\delta_{\widetilde t,t}\bigr),
\]
and therefore
\[
\frac12\bigl(h(t,\widetilde t)+h(\widetilde t,t)\bigr)\le 1.
\]
Verifying that $1 + \frac12\bigl(h(t,\widetilde t)-h(\widetilde t,t)\bigr)$ is an e-variable proves the claim.
\end{proof}

\subsubsection{2-Exchangeability optimality}
\label{app:exch_gro_proof}

\begin{proof}
The proof consists of a change-of-measure argument similar to that in~\citet{numeraire}. Indeed, denoting by $f_{T, \widetilde{T}}$ the density under the alternative $Q$ and by $\tilde{P}$ the measure with density
\[
p(T,\widetilde{T}) = \frac{1}{2}\bigl(f_{T, \widetilde{T}}(T,\widetilde{T}) + f_{T, \widetilde{T}}(\widetilde{T},T)\bigr)\,\mathbf{1}_{\{f_{T, \widetilde{T}}(T,\widetilde{T}) > 0\}},
\]
we obtain 
\begin{align*}
\mathbb{E}_Q\!\left[\frac{E}{E^\star}\right]
&= \mathbb{E}_Q\!\left[
\frac{1 + g(T,\widetilde{T})}
{\frac{f_{T, \widetilde{T}}(T,\widetilde{T})}{\frac{1}{2}\bigl(f_{T, \widetilde{T}}(T,\widetilde{T}) + f_{T, \widetilde{T}}(\widetilde{T},T)\bigr)}}
\right] \\
&= \mathbb{E}_{\tilde{P}}\!\left[(1 + g(T,\widetilde{T}))\,
\mathbf{1}_{\{f_{T, \widetilde{T}}(T,\widetilde{T}) > 0\}}\right] \\
&\le \mathbb{E}_{\tilde{P}}\!\left[1 + g(T,\widetilde{T})\right]
= 1.
\end{align*}

Note that there is no need for $E$ to be of the form $1+g$, since, by the very definition of an e-value, we have $\mathbb{E}_{\tilde{P}}[E]\leq 1$.

\end{proof}

\subsection{Intermediate null hypotheses proofs}
\subsubsection{Proof of Proposition \ref{prop:mx_as_cond_exch}}
\label{sect:proof_mx_as_cond_exch}
\begin{proof}
Let $P(X,Y,Z)$ be any distribution from the original representation of $\mathcal{P}^{\text{MX-CI}}_{K}$, and  $P'(\widetilde{X},X,Y,Z) := K_Z \otimes P(X,Y,Z)$, then by construction $P'(\widetilde{X}|X,Y,Z) \sim K_Z$, $X \indep \widetilde{X}|Y,Z$ and $P'(X|Y,Z) = P'(\widetilde{X}|Y,Z)$. Now, 
\begin{equation*}
   P'(\widetilde{X},X,Y,Z) = P(\widetilde{X}|X,Y,Z) \otimes P(X|Y,Z) \otimes P(Y,Z) =  K_Z \otimes K_Z \otimes P(Y,Z) = P'(X,\widetilde{X},Y,Z),
\end{equation*}
meaning that $P'$ is an element of \eqref{eq:mx_as_cond_exch}.

On the other hand, if we start with a $P'(\widetilde{X},X,Y,Z)$, from \eqref{eq:mx_as_cond_exch}, using exchangeability we immediately get $P'(X|\widetilde{X},Y,Z) = K_Z$ implying $X \indep_{P'} Y | Z$ and $P'(X|Z) = K_Z$, meaning that the marginal $P'(X,Y,Z)$ belongs to the original representation of the null.
\end{proof}

\subsubsection{Proof of GRO property for Table \ref{tab:null-distillation}}
\begin{proposition}
\label{prop:gro_property_proof}
The GRO e-variables for the null hypotheses in Table \ref{tab:null-distillation} are those provided in the table.
\end{proposition}
\begin{proof}
Any statistic $S$ that is $Q$-a.s. equal to the likelihood ratio of the point alternative distribution $Q$ and an element of the null hypothesis satisfies the following:\\
If $S$ is an e-variable for the null, then $S$ is a version of the GRO e-variable for testing against $Q$ \citep{numeraire}. All e-variables in Table \ref{tab:null-distillation} can be rewritten in such a form by extending the conditional likelihood ratio to a full likelihood ratio where necessary. Thus, all that remains is to check the e-variable property for each. The first null is treated in \citet{Grünwald02042024}, and the argument for the second is analogous:\\
Let $P \in \mathcal{P}^{\textnormal{PFX-CI}}_K$, then $P(T\mid Y,Z) = P(\widetilde{T}\mid Y,Z) = K^\#(\cdot\mid Y, Z)$, which has density function $f_{\widetilde{T}|Y,Z}$ by the model-X assumption on the alternative, and thus
\begin{equation*}
    \mathbb{E}_P\left[ \frac{f_{T|Y,Z}(T|Y,Z)}{f_{\widetilde{T}|Y,Z}(T|Y,Z)} \middle| Y,Z \right] = 1.
\end{equation*}
The e-variable property follows by the tower property of conditional expectations.\\
Exchangeability, conditional exchangeability and symmetry can all be proved using the same symmetry argument. We will explicitly show it for $B$-exchangeability only. For any exchangeable distribution $P(T,\widetilde{T}^B)$ and $(B+1)$-permutation $\sigma'$, we have:
\begin{align*}
    \mathbb{E}_{P(T,\widetilde{T}^B)} \left[ \frac{f_{T,\widetilde{T}^B}(T,\widetilde{T}^B)}{\frac{1}{(B+1)!} \sum_{\sigma} f_{T,\widetilde{T}^B}(\sigma(T,\widetilde{T}^B))} \right] &= \mathbb{E}_{P(\sigma'(T,\widetilde{T}^B))} \left[ \frac{f_{T,\widetilde{T}^B}(T,\widetilde{T}^B)}{\frac{1}{(B+1)!} \sum_{\sigma} f_{T,\widetilde{T}^B}(\sigma(T,\widetilde{T}^B))} \right] \\
    &= \mathbb{E}_{P(T,\widetilde{T}^B)} \left[ \frac{f_{T,\widetilde{T}^B}(\sigma'^{-1}(T,\widetilde{T}^B))}{\frac{1}{(B+1)!} \sum_{\sigma} f_{T,\widetilde{T}^B}(\sigma(T,\widetilde{T}^B))} \right]
\end{align*}
Thus, the proposed e-variable has the same expectation under $P$ for any permutation of its arguments. Summing over all $(B+1)!$ choices of $\sigma'$ gives us exactly the value $(B+1)!$, meaning that the original expectation must be exactly $1$.
\end{proof}

\subsubsection{Derandomization monotonically increases e-power}

\begin{proposition}
\label{prop:rb_monotonically_increases_power}
Let $T, \widetilde{T}^B$ all be conditionally independent given $(Y,Z)$, and $\widetilde{T}^B$ further identically distributed. Let $S(T,\widetilde{T})$ be a non-negative statistic and $S^B(T,\widetilde{T}^B):= \frac{1}{B}\sum_{i=1}^B S(T,\widetilde{T}_i)$, then the sequence

\begin{equation*}
    \mathbb{E}\left[\ln\left(S^B(T,\widetilde{T}^B)\right)\right]
\end{equation*}
is monotonically increasing in $B$.

\end{proposition}

\begin{proof}
Notice that 
\begin{equation*}
    S^B(T,\widetilde{T}^B) = \frac{1}{B} \sum _{i=1}^B S^{B-1}(T,\widetilde{T}^{B}_{-i}),
\end{equation*}
where $\widetilde{T}^{B}_{-i}$ is $\widetilde{T}^{B}$ without component $i$. Therefore, by Jensen's inequality 
\begin{equation*}
    \mathbb{E}\left[\ln\left(S^B(T,\widetilde{T}^B)\right)\right] \geq \frac{1}{B} \sum_{i=1}^B \mathbb{E}\left[\ln\left(S^{B-1}(T,\widetilde{T}^B_{-i})\right)\right] = \mathbb{E}\left[\ln\left(S^{B-1}(T,\widetilde{T}^{B-1})\right)\right].
\end{equation*}
In the last equality we used the conditional independence and i.i.d. property.
\end{proof}

\subsubsection{Proof for the limiting GRO e-variable for conditional exchangeability}
\label{sect:limiting_gro_cond_exch}
\begin{proof}
Assuming that under the point alternative $Q( \widetilde{T}^B,T,Y,Z)$ with density function $f$,  $T, \widetilde{T}^B$  are all conditionally independent given $(Y,Z)$, and that the $\widetilde{T}_i$ are conditionally identically distributed, we have:
\begin{align*}
    E^{\text{cond exch}}_B(t,\widetilde{t}^B,y,z) &= \frac{f_{T|Y,Z}(t|y,z)\prod_{i=1}^B f_{\widetilde{T}|Y,Z}(\widetilde{t}_i|y,z)}{\frac{1}{B+1} (f_{T|Y,Z}(t|y,z)\underset{i=1}{\overset{B}{\prod}} f_{\widetilde{T}|Y,Z}(\widetilde{t}_i|y,z) + f_{\widetilde{T}|Y,Z}(t|y,z) \overset{B}{\underset{j=1}{\sum}} f_{T|Y,Z}(\widetilde{t}_j|y,z)  \underset{i\neq j}{\prod} f_{\widetilde{T}|Y,Z}(\widetilde{t}_i|y,z))} \\
    &= \frac{f_{T|Y,Z}(t|y,z)}{\frac{1}{B+1} (f_{T|Y,Z}(t|y,z) + f_{\widetilde{T}|Y,Z}(t|y,z) \sum _{j=1}^B \frac{f_{T|Y,Z}(\widetilde{t}_j|y,z)}{f_{\widetilde{T}|Y,Z}(\widetilde{t}_j|y,z)}) }.
\end{align*}

First, we simply factorized the joint densities into products using conditional independence under the alternative, and noticed that, by symmetry, the sum of $(B+1)!$ terms simplifies into a smaller sum of $(B+1)$ terms. To get the second equality we divide everything by the product in the numerator.
Now, note that if we evaluate this e-variable in the random vector $(T, \widetilde{T}_1,...,\widetilde{T}_B)$ and let $B$ tend to infinity, by an application of the strong law of large numbers in the denominator, the limiting random variable is $Q$-a.s. equal to 
\begin{equation*}
    \frac{f_{T|Y,Z}(T|Y,Z)}{f_{\widetilde{T}|Y,Z}(T|Y,Z) \cdot \mathbb{E}_{Q}\left[ \frac{f_{T|Y,Z}(\widetilde{T}|Y,Z)}{f_{\widetilde{T}|Y,Z}(\widetilde{T}|Y,Z)} \middle|Y,Z \right]} = \frac{f_{T|Y,Z}(T|Y,Z)}{f_{\widetilde{T}|Y,Z}(T|Y,Z)}.
\end{equation*}
\end{proof}

\subsubsection{E-power sequence of conditional exchangeability e-variables}

\begin{proposition}
\label{prop:cond_exch_e_power_seq}
The sequence of e-powers
\begin{equation*}
    \mathbb{E}_Q\left[\ln\left(E^{\textnormal{cond exch}}_B\right)\right]
\end{equation*}
is (weakly) increasing in $B$ with limit 
\begin{equation*}
    \mathbb{E}_{Q(Y,Z)}\left[D_{\text{KL}}\left(Q(T|Y,Z) || Q(\widetilde{T}|Y,Z)\right)\right],
\end{equation*}
whenever the latter expectation is finite.
\end{proposition}

\begin{proof}
 $E^{\text{cond exch}}_B$ is an e-variable for $\mathcal{P}^{\text{cond exch}}_{B+1}$, while $E^{\text{cond exch}}_{B+1}$ is the GRO e-variable for the same null, and thus has a greater or equal e-power. Therefore, the e-power sequence weakly grows. Furthermore, in Section \ref{sect:limiting_gro_cond_exch}, we have seen that the limiting e-variable is the conditional likelihood ratio
\begin{equation*}
    E^{\textnormal{cond exch}}_{\infty} = \frac{f_{T|Y,Z}(T|Y,Z)}{f_{\widetilde{T}|Y,Z}(T|Y,Z)}.
\end{equation*}
Note that 
\begin{equation*}
    \mathbb{E}_Q\left[\ln\left(E^{\textnormal{cond exch}}_B\right)\right] \leq \mathbb{E}_{Q(Y,Z)}\left[D_{\text{KL}}\left(Q(T|Y,Z) || Q(\widetilde{T}|Y,Z)\right)\right]
\end{equation*}
must hold for all $B$. This is true because $\mathcal{P}^{\textnormal{PFX-CI}}_{K^{\#}}$ can always be equivalently represented (similar to Section \ref{sect:proof_mx_as_cond_exch} ) by adding $B$ conditionally independent dummy variables $\widetilde{T}_i \sim K^{\#}_{Y,Z}$ with same distribution under the alternative, and making the same exchangeability requirements as in $\mathcal{P}^{\textnormal{cond exch}}_B$. The addition of dummy variables does not change the GRO e-variable, but makes it immediately clear that $\mathcal{P}^{\textnormal{cond exch}}_B$ is a larger null hypothesis, and thus has a weaker GRO e-variable.
\\
What remains is to show that the e-power gap tends to $0$. Fatou's lemma implies:
\begin{equation*}
    \mathbb{E}\left[(\ln E^{\textnormal{cond exch}}_{\infty})^+\right] \leq \lim \inf_B \mathbb{E}\left[(\ln E^{\textnormal{cond exch}}_B)^+\right].
\end{equation*}
If we can further show that $(\ln E^{\textnormal{cond exch}}_B)^-$ is a uniformly integrable sequence, we will have
\begin{equation*}
    \mathbb{E}\left[(\ln E^{\textnormal{cond exch}}_{\infty})^-\right] = \lim_B \mathbb{E}\left[(\ln E^{\textnormal{cond exch}}_B)^-\right]
\end{equation*}
and the proof will be finished by subtracting this equality from Fatou's inequality. The rest of the proof is devoted to showing uniform integrability indeed holds.\\
Note that
\[
E^{\textnormal{cond exch}}_B = \frac{\frac{f_{T\mid Y,Z}(T\mid Y,Z)}{f_{\widetilde{T}\mid Y,Z}(T\mid Y,Z)}}{M_B}, \text{ with }
\qquad
M_B = \frac{1}{B+1}\left(\frac{f_{T\mid Y,Z}(T\mid Y,Z)}{f_{\widetilde{T}\mid Y,Z}(T\mid Y,Z)}
+ \sum_{i=1}^{B} \frac{f_{T\mid Y,Z}(\widetilde{T}_i\mid Y,Z)}{f_{\widetilde{T}\mid Y,Z}(\widetilde{T}_i\mid Y,Z)}\right).
\]
Since
\[
\ln E^{\textnormal{cond exch}}_B = \ln \frac{f_{T\mid Y,Z}(T\mid Y,Z)}{f_{\widetilde{T}\mid Y,Z}(T\mid Y,Z)} - \ln M_B,
\]
we have
\[
\big(\ln E^{\textnormal{cond exch}}_B\big)^- \le \left(\ln \frac{f_{T\mid Y,Z}(T\mid Y,Z)}{f_{\widetilde{T}\mid Y,Z}(T\mid Y,Z)}\right)^- + (\ln M_B)^+ .
\]
Define
\[
A_B = \frac{1}{B+1}\sum_{i=1}^{B} \frac{f_{T\mid Y,Z}(\widetilde{T}_i\mid Y,Z)}{f_{\widetilde{T}\mid Y,Z}(\widetilde{T}_i\mid Y,Z)},
\qquad\text{so that}\qquad
M_B \le \frac{f_{T\mid Y,Z}(T\mid Y,Z)}{f_{\widetilde{T}\mid Y,Z}(T\mid Y,Z)} + A_B .
\]
Using $\ln^+(a+b) \le \ln 2 + \ln^+ a + \ln^+ b$ and $\ln^+ x \le x$, we obtain
\[
\big(\ln E^{\textnormal{cond exch}}_B\big)^- \le \left|\ln \frac{f_{T\mid Y,Z}(T\mid Y,Z)}{f_{\widetilde{T}\mid Y,Z}(T\mid Y,Z)}\right| + \ln 2 + A_B .
\]
It therefore suffices to show that the right-hand side is uniformly integrable in $B$.

\emph{The term $A_B$.} Since $\widetilde{T}_1,\dots,\widetilde{T}_B$ are i.i.d.\ given $(Y,Z)$, the pairs $(\widetilde{T}_i,(Y,Z))$ share a common law, so the variables
\[
\frac{f_{T\mid Y,Z}(\widetilde{T}_i\mid Y,Z)}{f_{\widetilde{T}\mid Y,Z}(\widetilde{T}_i\mid Y,Z)}, \qquad i \ge 1,
\]
are identically distributed. They are nonnegative with
\[
\mathbb{E}\left[\frac{f_{T\mid Y,Z}(\widetilde{T}_i\mid Y,Z)}{f_{\widetilde{T}\mid Y,Z}(\widetilde{T}_i\mid Y,Z)} \,\middle|\, Y,Z\right]
= \int f_{T\mid Y,Z}(t\mid Y,Z)\,\mathrm{d}t = 1 .
\]
An identically distributed family of integrable variables is uniformly integrable, and so is the family of its averages. Thus, $\{A_B\}_{B\ge 1}$ is uniformly integrable.

\emph{The term involving $T$.} This is a single random variable, so it suffices that it is integrable under $f$. For the negative part, using $\ln x \leq x$,
\[
\mathbb{E}\left[\left(\ln \frac{f_{T\mid Y,Z}(T\mid Y,Z)}{f_{\widetilde{T}\mid Y,Z}(T\mid Y,Z)}\right)^- \,\middle|\, Y,Z\right]
\le \int_{\{t:f_{T\mid Y,Z}(t|Y,Z) < f_{\widetilde{T}\mid Y,Z}(t|Y,Z)\}} f_{\widetilde{T}\mid Y,Z}(t\mid Y,Z) \mathrm{d}t \leq 1 .
\]
For the positive part, integrability holds if and only if
\[
\mathbb{E}\left[\ln \frac{f_{T\mid Y,Z}(T\mid Y,Z)}{f_{\widetilde{T}\mid Y,Z}(T\mid Y,Z)}\right]
< \infty,
\]
which is true by assumption.

\end{proof}

\subsubsection{Limiting GRO e-variable for exchangeability}

The purpose of the proposition is to show that under regularity assumptions, a $Q$-a.s. limit of $E^{\textnormal{exch}}_B$ exists, and it has a very natural form. It is the conditional likelihood ratio of $T$ and $\widetilde{T}$, but given a function $\varphi(Y,Z)$ now, instead of the full pair $(Y,Z)$ -- which as we have seen in Section \ref{sect:limiting_gro_cond_exch} gives the limiting e-variable of $E^{\textnormal{cond exch}}_B$. The function $\varphi(Y,Z)$ essentially encodes all information that we can find out about the realized values of $(Y,Z)$ from an infinite i.i.d. sample $\widetilde{T}^{\infty}$ from $Q(\widetilde{T}|Y,Z)$. The regularity assumptions are strong and were chosen to provide a relatively simple proof of convergence rather than a very general result, and thus can likely be weakened.\\

\begin{proposition}
\label{prop:gro_for_exch_limit}
Suppose that under the point alternative $Q(\widetilde{T}^{\infty}, T, Y, Z)$, $\{\widetilde{T}_i\}_i, T$ are all independent given $Y,Z$, and $\{\widetilde{T}_i\}_i$ are further conditionally identically distributed. Furthermore, distributions $Q(T| Y, Z), Q(\widetilde{T}| Y, Z) , Q(Y, Z)$ are given by (conditional) densities $f_{T|Y,Z}, f_{\widetilde{T}|Y,Z}, f_{Y,Z}$ respectively, and $(Y,Z)$ takes values in a compact subspace of $\mathbb{R}^{d+1}$.
Write $w(t\mid y,z) \overset{def}{=} f_{T|Y,Z}(t\mid y,z)/f_{\widetilde{T}|Y,Z}(t\mid y,z)$ for the likelihood ratio (with $0/0=0$), and $\varphi(y,z) \overset{def}{=}  Q(\widetilde{T} \in \cdot|Y=y,Z=z)$.
Under the regularity assumptions
\begin{enumerate}
    \item \textbf{(Continuity and upper bound)}
          For every $t$, the maps $(y,z) \mapsto f_{T|Y,Z}(t\mid y,z)$ and
          $(y,z) \mapsto f_{\widetilde{T}|Y,Z}(t\mid y,z)$ are continuous, and there
          exists a constant $C < \infty$ s.t.
          $f_{T|Y,Z}(t \mid y,z), f_{\widetilde{T}|Y,Z}(t \mid y,z) \le C$ for all $t,y,z$.
 
    \item \textbf{(Lipschitz likelihood ratio)}
          $f_{T|Y,Z}(t\mid y,z)=0$ whenever $f_{\widetilde{T}|Y,Z}(t\mid y,z)=0$, and there is a function $L(t)\ge0$ such that for all $t,y,z,y',z'$,
          \[
              \bigl|w(t\mid y,z)-w(t\mid y',z')\bigr| \le L(t)||(y,z) - (y',z')||,
          \]
          and $\ell(y,z)\overset{def}{=}\int L(t)\, f_{\widetilde{T}|Y,Z}(t\mid y,z)\,dt<\infty$ for all $y,z$.
 
    \item \textbf{(Posterior consistency)}
          $Q(\widetilde{T}^\infty,Y,Z)$-a.s.,
          \[
              f_{Y,Z\mid \widetilde{T}^B}(\cdot,\cdot \mid \widetilde{T}^B)
              \xrightarrow{w} \Pi_{\varphi(Y,Z)}
              \quad \text{as } B \to \infty,
          \]
          where $\Pi_{\varphi(Y,Z)}$ is a probability measure supported on the
          fiber $\varphi^{-1}(\{\varphi(Y,Z)\})$.
\end{enumerate}
the limiting e-variable for conditional exchangeability is $Q(\widetilde{T}^\infty, T ,Y,Z)$-a.s.\ given by
\begin{equation*}
 \lim _{B \to \infty}E_B^{\textnormal{exch}}(T,\widetilde{T}^B) = \frac{f_{T|\varphi(Y,Z)}(T\mid\varphi(Y,Z))}{  f_{\widetilde{T}|Y,Z}(T\mid Y,Z)},
\end{equation*}
where
$f_{T|\varphi(Y,Z)}(t\mid\varphi(Y,Z)) \overset{def}{=} \int f_{T|Y,Z}(t\mid y,z)\,\Pi_{\varphi(Y,Z)}(dy,dz)$.
\end{proposition}

\begin{proof}
Throughout, $(Y,Z)$ denotes the realized value, so that given $(Y,Z)$ the $\widetilde{T}_j$ are i.i.d.\ with density $f_{\widetilde{T}|Y,Z}(\cdot\mid Y,Z)$ and $T$ has density $f_{T|Y,Z}(\cdot\mid Y,Z)$, while lowercase $(y,z)$ is an integration variable. By Assumption 2, $f_{T|Y,Z}(t\mid y,z)=w(t\mid y,z)\,f_{\widetilde{T}|Y,Z}(t\mid y,z)$ for all $t,y,z$.

\textbf{Step 1 (rewriting the e-variable).}
Using the same product decomposition as in $E_B^{\text{cond exch}}$ (see Appendix \ref{sect:limiting_gro_cond_exch}), but marginalizing over $y$ and $z$,

\begin{align*}
     &E_B^{\text{exch}}(t,\widetilde{t}^B) \\
     &=\frac{(B+1)\int f_{T|Y,Z}(t\mid y,z)\prod_{i=1}^B f_{\widetilde{T}|Y,Z}(\widetilde{t}_i\mid y,z)\, f_{Y,Z}(y,z)\,dy\,dz}{
      \int \Bigl[f_{T|Y,Z}(t\mid y,z)\overset{B}{\underset{i=1}{\prod}} f_{\widetilde{T}|Y,Z}(\widetilde{t}_i\mid y,z)
     +  f_{\widetilde{T}|Y,Z}(t\mid y,z) \overset{B}{\underset{j=1}{\sum}}f_{T|Y,Z}(\widetilde{t}_j\mid y,z) \underset{i\neq j}{\prod} f_{\widetilde{T}|Y,Z}(\widetilde{t}_i\mid y,z)\Bigr] f_{Y,Z}(y,z)\,dy\,dz.}
\end{align*}

Since $f_{T|Y,Z}(\widetilde{t}_j\mid y,z)\prod_{i\neq j}f_{\widetilde{T}|Y,Z}(\widetilde{t}_i\mid y,z)=w(\widetilde{t}_j\mid y,z)\prod_{i}f_{\widetilde{T}|Y,Z}(\widetilde{t}_i\mid y,z)$, dividing numerator and denominator by $(B+1)\int\prod_{i}f_{\widetilde{T}|Y,Z}(\widetilde{t}_i\mid y,z)f_{Y,Z}(y,z)\,dy\,dz$ and using Bayes' rule gives
\begin{equation}\label{eq:EB-rewritten}
    E_B^{\text{exch}}(t,\widetilde{t}^B)=\frac{A_B(t)}{\frac{1}{B+1}A_B(t)+D_B(t)},
\end{equation}
where
\begin{align*}
    A_B(t)&=\int f_{T|Y,Z}(t\mid y,z)\,f_{Y,Z\mid\widetilde{T}^B}(y,z\mid\widetilde{t}^B)\,dy\,dz,\\
    D_B(t)&=\int f_{\widetilde{T}|Y,Z}(t\mid y,z)\left(\frac{1}{B+1}\sum_{j=1}^B w(\widetilde{t}_j\mid y,z)\right)f_{Y,Z\mid\widetilde{T}^B}(y,z\mid\widetilde{t}^B)\,dy\,dz.
\end{align*}
From now on we evaluate at $\widetilde{t}^B=\widetilde{T}^B$.
 
\textbf{Step 2 ($A_B$-term).}
For fixed $t$, the function $(y,z)\mapsto f_{T|Y,Z}(t\mid y,z)$ is continuous and bounded by $C$ (Assumption 1). Hence, by the definition of weak convergence and Assumption 3,
\[
    A_B(t)\to\int f_{T|Y,Z}(t\mid y,z)\,\Pi_{\varphi(Y,Z)}(dy,dz)=f_{T|\varphi(Y,Z)}(t\mid\varphi(Y,Z)),
    \qquad\text{and}\qquad \frac{A_B(t)}{B+1}\le\frac{C}{B+1}\to0 .
\]
 
\textbf{Step 3 (the likelihood-ratio average converges uniformly in $(y,z)$).}
Let
\[
    \overline{w}_B(y,z)\overset{ \mathrm{def}}{=}\frac1B\sum_{j=1}^B w(\widetilde{T}_j\mid y,z),
    \qquad
    r(y,z)\overset{\mathrm{def}}{=}\int w(t'\mid y,z)\,f_{\widetilde{T}|Y,Z}(t'\mid Y,Z)\,dt',
\]
so that $r(y,z)$ is the mean of $w(\widetilde{T}_j\mid y,z)$. By Assumption 2,
$|r(y,z)-r(y',z')|\le \ell(Y,Z)\, ||(y,z) - (y',z')||$. Moreover $r(Y,Z)=\int f_{T|Y,Z}(t'\mid Y,Z)\,dt'=1$, so $r$ is bounded by $1+\ell(Y,Z)\cdot\operatorname{diam}<\infty$. We show that, almost surely,
\begin{equation}\label{eq:ULLN}
    \sup_{y,z}\,\bigl|\overline{w}_B(y,z)-r(y,z)\bigr|\to0 .
\end{equation}
Fix $\delta>0$. By compactness, the space of $(y,z)$ is covered by finitely many balls of radius $\delta$, with centres $(y_1,z_1),\dots,(y_K,z_K)$. By the strong law of large numbers (applied to finitely many i.i.d.\ sequences with finite means), almost surely
\[
    \overline{w}_B(y_k,z_k)\to r(y_k,z_k)\ \ (k=1,\dots,K),
    \qquad
    \frac1B\sum_{j=1}^B L(\widetilde{T}_j)\to\ell(Y,Z) .
\]
For any $(y,z)$, choose $k$ with $||(y,z)-(y_k,z_k)||<\delta$. Using Assumption 2 twice,
\begin{align*}
    \bigl|\overline{w}_B(y,z)-r(y,z)\bigr|
    &\le \bigl|\overline{w}_B(y,z)-\overline{w}_B(y_k,z_k)\bigr| + \bigl|\overline{w}_B(y_k,z_k)-r(y_k,z_k)\bigr| + \bigl|r(y_k,z_k)-r(y,z)\bigr| \\
    &\le \delta\,\frac1B\sum_{j=1}^B L(\widetilde{T}_j) + \max_{k\le K}\bigl|\overline{w}_B(y_k,z_k)-r(y_k,z_k)\bigr| + \delta\,\ell(Y,Z) .
\end{align*}
The right-hand side does not depend on $(y,z)$, and its $\limsup$ is $2\delta\,\ell(Y,Z)$. Since $\delta$ was arbitrary, we get \eqref{eq:ULLN}.
 
\textbf{Step 4 ($D_B$-term).}
Since $\frac{1}{B+1}\sum_{j=1}^B w(\widetilde{T}_j\mid y,z)=\frac{B}{B+1}\overline{w}_B(y,z)$, we can split
\begin{multline*}
    \frac{B+1}{B}\,D_B(t)=\int f_{\widetilde{T}|Y,Z}(t\mid y,z)\,r(y,z)\,f_{Y,Z\mid\widetilde{T}^B}(y,z\mid\widetilde{T}^B)\,dy\,dz\\
    +\int f_{\widetilde{T}|Y,Z}(t\mid y,z)\bigl(\overline{w}_B(y,z)-r(y,z)\bigr)\,f_{Y,Z\mid\widetilde{T}^B}(y,z\mid\widetilde{T}^B)\,dy\,dz.
\end{multline*}
The second term is at most $C\sup_{y,z}|\overline{w}_B(y,z)-r(y,z)|$ in absolute value, which tends to $0$ by \eqref{eq:ULLN}. In the first term, $(y,z)\mapsto f_{\widetilde{T}|Y,Z}(t\mid y,z)\,r(y,z)$ is continuous and bounded (Assumption 1 and Step 3), so by Assumption 3 it converges to $\int f_{\widetilde{T}|Y,Z}(t\mid y,z)\,r(y,z)\,\Pi_{\varphi(Y,Z)}(dy,dz)$. For every $(y,z)$ in the fiber, $Q(\widetilde{T}|Y=y,Z=z)=Q(\widetilde{T}|Y=Y,Z=Z)$, and therefore
\[
    f_{\widetilde{T}|Y,Z}(t\mid y,z)=f_{\widetilde{T}|Y,Z}(t\mid Y,Z),
    \qquad
    r(y,z)=\int w(t'\mid y,z)\,f_{\widetilde{T}|Y,Z}(t'\mid y,z)\,dt'=\int f_{T|Y,Z}(t'\mid y,z)\,dt'=1 .
\]
Since $\Pi_{\varphi(Y,Z)}$ is supported on the fiber, $D_B(t)\to f_{\widetilde{T}|Y,Z}(t\mid Y,Z)$.
 
\textbf{Step 5 (conclusion).}
The almost-sure events used in Steps 2--4 (Assumption 3, and the countably many strong laws in Step 3) do not depend on $t$. Hence, almost surely, Steps 2 and 4 hold simultaneously for all $t$, and in particular for $t=T$. Moreover $f_{\widetilde{T}|Y,Z}(T\mid Y,Z)>0$ a.s., because $T$ has density $f_{T|Y,Z}(\cdot\mid Y,Z)$, which vanishes wherever $f_{\widetilde{T}|Y,Z}(\cdot\mid Y,Z)$ does (Assumption 2). Plugging Steps 2 and 4 into \eqref{eq:EB-rewritten},
\begin{equation*}
    \lim_{B \to \infty} E_B^{\text{exch}}(T,\widetilde{T}^B)
    = \frac{\int f_{T|Y,Z}(T\mid y,z)\, \Pi_{\varphi(Y,Z)}(dy,dz)}{ f_{\widetilde{T}|Y,Z}(T\mid Y,Z) }
    = \frac{f_{T|\varphi(Y,Z)}(T\mid\varphi(Y,Z))}{  f_{\widetilde{T}|Y,Z}(T\mid Y,Z)}.
    \qedhere
\end{equation*}
\end{proof}

\subsection{Type-I error inflation}

\subsubsection{From approximate e-variables to approximate type-I error (Proposition \ref{prop:approx_e_type_I})}

\begin{proof}
Define the cumulative inflated-controlled event
$G_n
:=
\bigcap_{t=1}^n A_n.$ 
By the definition of the approximate e-process,
\begin{align*}
\mathbb{E}_{P}
\left[
E_n\ind_{G_n}
\middle|
\mathcal{F}_{n-1}
\right]
=
\ind_{G_{n-1}}
\mathbb{E}_{P}
\left[
E_n\ind_{A_n}
\middle|
\mathcal{F}_{n-1}
\right]
\leq
(1+\epsilon_n)
\ind_{G_{n-1}}E_{n-1}.
\end{align*}

Define
\begin{align*}
M_n
:=
\frac{
E_{n}\ind_{G_n}
}{
\prod_{t=1}^n(1+\epsilon_n)
}.
\end{align*}
Then $(M_n)_{n\geq0}$ is a nonnegative supermartingale with $M_0=1$. By Ville's inequality,
\begin{align*}
P
\left(
\sup_{1\leq n\leq N}M_n
\geq \frac{1}{\alpha}
\right)
\leq \alpha.
\end{align*}

Let
\begin{align*}
C_N
:=
\prod_{t=1}^N(1+\epsilon_n).
\end{align*}
On the event $G_N$, for every $n\leq N$,
\begin{align*}
E_{n}
&\leq
C_N M_n.
\end{align*}
Consequently,
\begin{align*}
\left\{
\sup_{1\leq n\leq N}
E_{n}
\geq \frac{1}{\alpha}
\right\}\cap G_N
&\subseteq
\left\{
\sup_{1\leq n\leq N}
M_n
\geq
\frac{1}{\alpha C_N}
\right\}.
\end{align*}
Applying Ville's inequality at level $\alpha C_N$ gives
\begin{align*}
P
\left(
\sup_{1\leq n\leq N}
E_n
\geq \frac{1}{\alpha},
\ G_N
\right)
&\leq
\alpha C_N.
\end{align*}

It remains to control the probability that at least one of the events $A_n$ fails. By the probability assumption on the approximate e-process definition and union bound,
\begin{align*}
P(G_N^c)
=
P
\left(
\bigcup_{n=1}^N A_n^c
\right)
\
\leq
\sum_{n=1}^N
P(A_n^c)
\
\leq
\sum_{n=1}^N\delta_n.
\end{align*}
Therefore,
\begin{align*}
P
\left(
\sup_{1\leq n\leq N}
E_n
\geq \frac{1}{\alpha}
\right)
&\leq
\sum_{n=1}^N\delta_n
+
\alpha
\prod_{t=1}^N(1+\epsilon_n).
\end{align*}

Finally, since $\log(1+x)\leq x$ for all $x\geq0$,
\begin{align*}
\prod_{t=1}^N(1+\epsilon_n)
=
\exp\left(
\sum_{n=1}^N\log(1+\epsilon_n)
\right)\leq
\exp\left(
\sum_{n=1}^N\epsilon_n
\right).
\end{align*}
This proves the result.
\end{proof}

\subsubsection{Double robustness}

\begin{proof}
We prove the three statements separately.

\paragraph{(i) Lipschitz convergence.}

%

By the Kantorovich--Rubinstein dual representation of the Wasserstein distance, we have that for every Lipschitz function $f$

$$
\left|
\int f\,dQ-\int f\,dP
\right|
\leq
\operatorname{Lip}(f)\mathcal W_1(Q,P).
$$

By assumption, $x\mapsto r_n(x,z)$
is $b_n$-Lipschitz uniformly on $z, y$. Therefore, for any $z, y$, we have that

\begin{align*}
\left|
\mathbb E[r_n(\widetilde X'_n,z, y)]
-
\mathbb E[r_n(X,z, y)]
\right|
&\leq
b_n(z, y)\mathcal W_1(K_n,K)
\\
&=:
a_nb_n.
\end{align*}


\paragraph{(ii) Uniform convergence.} Fix $z$ and $y$. Since $
\sup_x|r_n(x,z, y)|\leq b_n(z, y),
$
we have
$
\|r_n(\cdot,z, y)\|_\infty\leq b_n.
$

Total variation gives
$$
\left|
\int f\,dQ-\int f\,dP
\right|
\leq
2\|f\|_\infty
\mathrm d_{\mathrm{TV}}(P,Q).
$$

Applying this to $f=r_n(,\cdot,z)$, we obtain

$$
\begin{aligned}
\left|
\mathbb E[r_n(\widetilde X'_t,z, y)]
-
\mathbb E[r_n(X,z, y)]
\right|
&\leq
2b_n(z, y)\mathrm d_{\mathrm{TV}}(K_n, K)
\\
&=
2a_nb_n(z, y).
\end{aligned}
$$

\paragraph{(iii) $L_2$ convergence.} the chi-squared divergence $\mathrm d_{\chi^2}$ is defined as

$$
\mathrm d_{\chi^2}(P\Vert Q)
:=
\left(
\int
\left(
\frac{dP}{dQ}-1
\right)^2dQ
\right)^{1/2},
\qquad P\ll Q.
$$

Suppose that $
K_n\ll K
$ and let $(X, Z, Y)$ sampled from $P_\star(X, Z, Y)$ and $(\widetilde{X}', Z, Y)$ from $P_n(\widetilde{X}', Z, Y)$. Denote by

$$
D_n(x,z, y)
:=\frac{dP_n(\widetilde{X}', Z, Y)}{dP_\star(X, Z, Y)}(x, z, y)=
\frac{dK_n(\,\cdot\mid z)}
     {dK(\,\cdot\mid z)}(x),
$$
the Radon-Nikodym derivative.

Then $\mathbb E_{K_n}[r_n]
=
\mathbb E_{K}[D_nr_n].
$ Therefore, we have that,

$$
\mathbb E_{P_\star}[r_n]-\mathbb E_{P_n}[r_n]
=
\mathbb E_{P_\star}[(D_n-1)r_n].
$$

By the Cauchy--Schwarz inequality and the total expectation law, we obtain that

\begin{align*}
\left|
\mathbb E_{P_t}[r_n]-\mathbb E_{P_\star}[r_n]
\right|
&\leq
\left(
\mathbb E_{P_\star}\left[(D_t-1)^2\right]
\right)^{1/2}
\left(
\mathbb E_{P_\star}\left[r_n^2\right]
\right)^{1/2}
\\
&\leq
\left(
\mathbb E_{P_\star(Z)}\left[\mathbb{E}_{K}\left[(D_t-1)^2\right]\right]
\right)^{1/2}
\left(
\mathbb E_{P_\star}[r_n^2]
\right)^{1/2}
\\
&=
\sqrt{\mathbb{E}\left[\mathrm d_{\chi^2}(K_n\Vert K)^2\right]}
\,
\|r_n\|_{L_2}.
\end{align*}

By assumption, $\|r_n\|_{L_2}
\leq b_n.$ Hence, we conclude that

$$
\left|
\mathbb E_{P_t}[r_n]-\mathbb E_{P_\star}[r_n]
\right|
\leq
a_nb_n.
$$

\end{proof}

\subsubsection{Coin-betting type-I error inflation}

\begin{proof}
Fix $n$ and condition on $\mathcal F_{n-1}$, $Z_n$ and $Y_n$. Note that under the null, $
X_n\mid Z_n,Y_n,\mathcal F_{n-1}
\sim
K(\cdot\mid Z_n),
$
because $X_n\indep Y_n\mid Z_n$. By construction, we have that 
$
\widetilde X_n\mid Z_n,Y_n,\mathcal F_{n-1}
\sim
K_n(\cdot\mid Z_n).
$

Therefore, we have that
\begin{align*}
\mathbb E
\left[
E_n
\mid
\mathcal F_{n-1},Z_n,Y_n
\right]=
1+\lambda
\mathbb E_{K}
\left[
l(\widehat m_n(X,Z_n),Y_n)
\mid
\mathcal F_{n-1},Z_n,Y_n
\right]
-\lambda
\mathbb E_{K_n}
\left[
l(\widehat m_n(\widetilde X_n,Z_n),Y_n)
\mid
\mathcal F_{n-1},Z_n,Y_n
\right].
\end{align*}

Introduce the $X$-independent reference quantity $l_0(Z_n,Y_n):=l(m_0(Z_n),Y_n).$ Hence it has exactly the same value under $K$ and $K_n$, and therefore cancels from the difference. Thus
\begin{align*}
\mathbb E &
\left[
E_n
\mid
\mathcal F_{n-1},Z_n,Y_n
\right]-1
\\&=
\lambda
\Bigg[
\mathbb E_{K}
\left[
l(\widehat m_n(X,Z_n),Y_n)-l_0(Z_n,Y_n)
\mid
\mathcal F_{n-1},Z_n,Y_n
\right]
-
\mathbb E_{K_n}
\left[
l(\widehat m_n(\widetilde X_n,Z_n),Y_n)-l_0(Z_n,Y_n)
\mid
\mathcal F_{n-1},Z_n,Y_n
\right]
\Bigg].
\end{align*}

The expression in braces is precisely the difference between the expectations of the model-dependent component of the loss under the true and estimated conditional distributions. By Proposition~\ref{prop:double_robust}, on the corresponding good event $A_n$,
$$
\left|
\mathbb E_{K}
\left[
l(\widehat m_n(X,Z_n),Y_n)-l_0(Z_n,Y_n)
\right]
-
\mathbb E_{K_n}
\left[
l(\widehat m_n(\widetilde X_n,Z_n),Y_n)-l_0(Z_n,Y_n)
\right]
\right|
\leq
a_nb_n.
$$

Consequently,

$$
\mathbb E[E_t\mid\mathcal F_{n-1}]
\leq
1+\lambda a_nb_n.
$$

Thus, setting

$$
\epsilon_n:=\lambda a_nb_n,
$$

the process $(E_t)_{t\geq1}$ is an approximate e-process in the sense of Definition~\ref{def:approx_e_process}, with one-step inflation $\epsilon_n$.

We can therefore apply Proposition~\ref{prop:approx_e_type_I}. This gives
\begin{align*}
\mathbb P
\left(
\sup_{1\leq n\leq N}
E^{(n)}
\geq \frac{1}{\alpha}
\right)
\leq
\sum_{n=1}^N\delta_n
+
\alpha
\exp\left(
\sum_{n=1}^N\epsilon_n
\right)
=
\sum_{n=1}^N\delta_n
+
\alpha
\exp\left(
\lambda\sum_{n=1}^Na_nb_n
\right).
\end{align*}
Subtracting $\alpha$ from both sides yields
\begin{align*}
\mathbb P
\left(
\sup_{1\leq n\leq N}
E^{(n)}
\geq \frac{1}{\alpha}
\right)
-\alpha
&\leq
\sum_{n=1}^N\delta_n
+
\alpha
\left[
\exp\left(
\lambda\sum_{n=1}^Na_nb_n
\right)-1
\right].
\end{align*}

\end{proof}

\section{Experiments}
\label{appendix:experiments}

\subsection{Estimating the conditional distribution}\label{sect:cond_dist}

In practice, as motivated in Section \ref{sect:safety}, we usually do not have access to the conditional distribution of $X \mid Z$, nor can we sample from it directly. Nevertheless, there is substantial ongoing research addressing this issue (see e.g. \cite{barber2019robustinferenceknockoffs, zhu_deeplink,maxway}). When $X$ is discrete, it suffices to model the conditional class probabilities of $X$ given $Z$ and sample accordingly; this can be implemented, for instance, via logistic regression. 

When $X$ is continuous, the simplest option is to assume that the input is Gaussian given the remaining inputs. For this kind of parametric option, \citet{Berrett} proved convergence in total variation distance. For simplicity, we choose this option in our simulated data experiments. Nevertheless, a somewhat more general idea consists of first regressing $X$ on $Z$ and then generating samples by adding permuted residuals from the fitted regression model. This sampling idea is quite popular \citep{fisher2019all, chamma2024statistically, blain2025knockoffsfaildiagnosingfixing} since it works well in practice and is usually fast, as it only requires simple models. It generalizes the Gaussian case to homoskedastic distributions. For further discussion of the guarantees of this conditional sampling procedure, we refer to \citet{reyerolobo2025conditionalfeatureimportancerevisited}.

Finally, promising new approaches to conditional sampling could be based on foundation models \citep{qu2026tabiclv2}, which estimate the quantiles of the conditional distribution and are quite accurate for this task, as noted by \citet{izbicki2026benchmarkingtabularfoundationmodels}.


\subsection{Coin-betting can beat Exponential setting (Figure \ref{fig:exp_vs_cb})}\label{app:simulation}

In Figure~\ref{fig:exp_vs_cb} we illustrate the results from Theorem~\ref{thm:small_signal}. To do so, we consider the simple linear model
\[
Y = X + \gamma Z+\epsilon,
\qquad
X \sim \mathrm{Unif}([0,1]),
\qquad
\epsilon \sim B\times(\mathrm{Beta}(0.1, 0.1) - 0.5),
\]
for different values of \(B\). A $\mathrm{Beta}(0.1, 0.1)$ random variable is distributed in $[0, 1]$ concentrated near the boundaries. Shifting by $0.5$ centers the distribution, and multiplying by $B$ has the effect of scaling the noise.  In particular, we do not include an additional \(Z\) (i.e. we set $\gamma = 0$), and we sample \(\widetilde X \sim \mathrm{Unif}([0,1])\), i.e., \(K(X\mid Z)=\mathrm{Unif}([0,1])\). We take the model to be the oracle predictor \(m(X)=X\). To ensure boundedness of the loss, we use the clipped squared loss
\[
l(m(X),Y) = \min\bigl((X-Y)^2,1\bigr),
\]
which implies any \(\lambda \in [0,1]\) is a valid betting parameter.

On the left of Figure~\ref{fig:exp_vs_cb}, we plot the e-power, $\mathbb{E}[\log E^{\lambda^\star}_{\mathrm{cb}}]$ and $\mathbb{E}[\log E^{\eta^\star}_{\mathrm{exp}}]$, evaluated at the optimal $\lambda^\star$ and $\eta^\star$, which are selected over a large grid of possible values.
The expectations are approximated using Monte Carlo with 3000 samples. To simulate regimes with different effect sizes \(\Delta\), we vary the noise level via the parameter \(B\). When the noise level is small, the effect size is large, and it is beneficial to take larger bets; however, the coin-betting procedure is constrained by the requirement \(\lambda \leq 1\), whereas the exponential e-value can still increase through larger values of \(\eta\). In contrast, when \(B\) is large, the problem becomes harder and the effect size decreases. In this regime, consistent with the theorem, the coin-betting e-value can outperform the exponential e-value. Here we focus on the region in $B$ where $E^{\lambda^\star}_{\mathrm{cb}}$ overtakes $E^{\eta^\star}_{\mathrm{exp}}$ in e-power.

In the middle we highlight the difference of the e-powers shown on the left by plotting their ratio. 

On the right, we fix the noise level to a value in the small-signal regime ($B = 1.5$), where the coin-betting e-value has higher e-power than the exponential e-value. In this experiment, at each time step we use the true linear model and update the parameters $\eta$ and $\lambda$ in a prequential manner, as specified in Equations~\ref{eq:eta_optim} and~\ref{eq:lambda_optim}. We observe that the coin-betting evidence grows exponentially faster, indicating a clear advantage consistent with the gap in e-power shown on the left. 

\subsection{E-values considered in the experiments}\label{sect:e_values_anti_exps}

\begin{table}[ht]
  \centering
  \renewcommand{\arraystretch}{2.4}
  \caption{Summary of the considered e-values using as statistics $T, \widetilde T$ the losses $l \overset{\mathrm{def}}{=} l(m(X, Z), Y)$
and $\widetilde l \overset{\mathrm{def}}{=} l(m(\widetilde X, Z), Y)$}
  \label{tab:evalues}
  \begin{tabular}{@{}c l c@{}}
    \toprule
    & \textbf{Name} & \textbf{Form} \\
    \midrule
    \multirow{5}{*}[-4ex]{\rotatebox[origin=c]{90}{\textit{Antisymmetric}}}
    & coin-betting
      & $\displaystyle 1 + \frac{\lambda}{M}\Big( \big( (\tilde{l} - l) \wedge M \big) \vee (-M) \Big)$ \\
    & tanh
      & $\displaystyle 1 + \lambda \tanh\!\big( 20 (\tilde{l} - l) \big)$ \\
    & sign(\_preq)
      & $\displaystyle 1 + \lambda \operatorname{sign}(\tilde{l} - l)$ \\
    & exch-GRO
      & $\displaystyle \frac{2 f_{\tilde{l},l}(\tilde{l}, l)}
                           {f_{\tilde{l},l}(\tilde{l}, l) + f_{\tilde{l},l}(l, \tilde{l})}$ \\
    & sym-GRO
      & $\displaystyle \frac{2 f(\tilde{l} - l)}
                           {f(\tilde{l} - l) + f(l - \tilde{l})}$ \\
    \midrule
    & exponential
      & $\displaystyle \frac{e^{-\eta l}}{\mathbb{E}_{\tilde{X}\mid Z}\big[ e^{-\eta \tilde{l}} \big]}$ \\
    \bottomrule
  \end{tabular}
\end{table}

In Table \ref{tab:evalues} we show the analytical expressions of the e-variables used in the experiments. All the antisymmetric e-values can be derandomized by taking the expectation over $\widetilde{X}$ given $Z$.

 Except for \texttt{sign}, all hyperparameters are learned prequentially and optimized for e-power maximization.
 For the sign e-variables, we denote in the experiments as \texttt{sign\_preq} the e-process learned by optimizing $\lambda$ prequentially. 
 The \texttt{sign} method corresponds exactly to the procedure of \cite{shaer23a}, which averages over all $\lambda$ values in a grid (i.e. puts a uniform prior) instead of sequentially selecting the e-power maximizer.

\subsection{Simulation setting: e-power of antisymmetric vs exponential e-variables }\label{subsec:sim_setting}

\begin{figure}[H]
    \centering
    \includegraphics[width=1\linewidth]{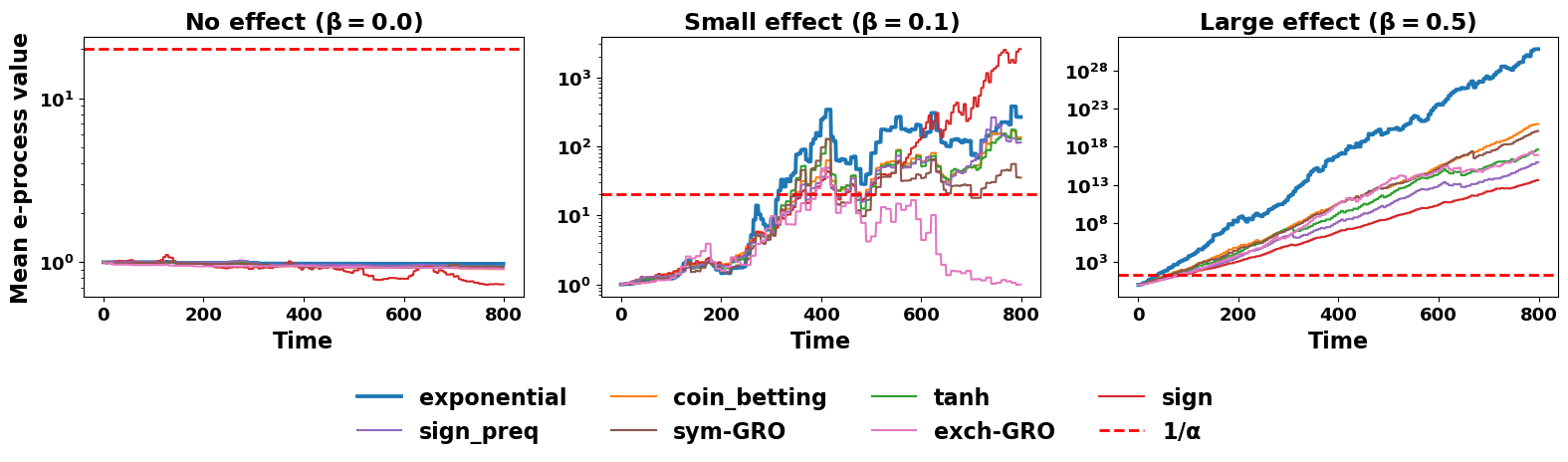}
    \caption{\textbf{Antisymmetric vs.\ exponential:} We consider the model
$
Y=(Z^\top W)^2+\beta X+\varepsilon.
$
 The parameter $\beta$ controls the predictive signal of $X$. Under a null effect ($\beta=0$), neither e-process rejects. For small signal strengths, antisymmetric sign e-process outperforms the exponential, whereas for larger signals, the exponential accumulates evidence fastest.
    }\label{fig:simulated_experiment}

\end{figure}

See Section \ref{sect:e_values_anti_exps} for a description of all e-values included in this simulation. We closely follow the simulation setting described  by \cite{shaer23a} in section 5. The data is generated according to:
\begin{equation*}
    Y=(Z^\top W)^2+\beta X+\varepsilon, \text{ } X\mid Z \sim \mathcal{N}(U^\top Z,1), \text{ } Z\sim\mathcal{N}(0,I_d),
\end{equation*}\label{eq:simulated_data_equation}
with $d=19$. The main difference compared to the original experiment is that here $U$ and $W$ are uniformly-sampled probability vectors of length $d$, as opposed to being sampled from a $d$-dimensional standard Gaussian. We introduce this change in order to decrease the variance of results across experiments, and to ensure that $\beta$ alone is a good representation of effect size in the experiment, which otherwise also depends on the magnitudes of the sampled vectors $U$ and $W$.\\
Although the true conditional distribution $X|Z$ is known exactly, this information is not fed to the e-processes. Instead, the model-X kernel is learned over time, showcasing practical safety guarantees and e-power. The employed model $\widehat{m}$ is Lasso regression. In order to avoid retraining the model after each new datapoint, a batch size $b$ is fixed, so both model retraining and e-process adaptation happens only after $b$ new datapoints are acquired. Before running the e-processes, the model $\widehat{m}$, model-X kernel, betting parameters and kernel density estimates for $E^{\text{exch}}$ and $E^{\text{sym}}$ (see Appendix \ref{appendix:hypotheses} for definitions) were tuned on $200$ pre-training datapoints. Furthermore, all e-variables are enhanced by derandomization over $B=5$ conditionally independent $\widetilde{X}$ samples from the learned model-X kernel. Kernel density estimation of the distribution of $(T, \widetilde{T}) = ((Y-\widehat{m} (X,Z))^2, (Y-\widehat{m} (\widetilde{X},Z)^2)$ is performed using only the first $\widetilde{X}$ sample for each $X$, and analogously for the distribution of $\widetilde{T}-T$. The experiment was performed independently on batches of size: [5, 10, 20], and over 20 random seeds. Figure \ref{fig:simulated_experiment} represents the average outcome over all seeds and batch sizes.

Under the null hypothesis ($\beta = 0$) no e-process rejects, confirming type-I error guarantees. The middle figure confirms that antisymmetric e-variables can outperform the exponential when the signal is small. The figure on the right-hand side shows that with large effects, the exponential e-variable dominates all others. Although $E^{\text{exch}}$ and $E^{\text{sym}}$ should theoretically dominate all other antisymmetric e-variables in the experiment, this is not observed because:
\begin{itemize}
    \item We use approximations of the true GRO e-variables $E^{\text{exch}}$ and $E^{\text{sym}}$, which require kernel density estimation, and thus might require larger (pratraining) datasets to acquire full e-power
    \item Derandomization  with $B=5$ samples is used to enhance the e-power of all e-variables, after which the growth-rate optimality of $E^{\text{exch}}$ and $E^{\text{sym}}$ no longer necessarily holds (see Appendix \ref{appendix:effects_of_derandomization} for more details on the effects of derandomization)
\end{itemize}

\subsection{Real data: HIV dataset (Figure \ref{fig:hiv_data})}\label{sect:hiv_data}

In this section we further discuss the experiment shown in the main text in Figure \ref{fig:hiv_data}, based on the HIV drug resistance dataset \footnote{Dataset available at \url{https://hivdb.stanford.edu/pages/published_analysis/genophenoPNAS2006/}} used by \citet{shaer23a} to showcase their e-values, and previously in the knockoff literature by \citet{barber2015controlling}, \citet{lu2018deeppink}, \citet{romano2020deep}, \citet{shaer2023learning}.
The task is to detect whether a mutation in the human immunodeficiency virus (HIV) is associated to an increase in drug resistance \citet{rhee2006genotypic}. The dataset contains in $d = 150$ binary features and $n = 1555$ real-valued responses. In the experiment we focus on three of these features that were deemed representative by \citet{shaer23a}: one mutation that has not been shown to affect drug resistance (12S), one mutation that has a major effect on drug resistance but for which it is difficult to collect evidence (48M), and one that has a major effect on drug resistance (47V).

In Figure \ref{fig:hiv_data} we show for three different mutations the evidence accumulated by different e-variables against the conditional independence null. 
The plotted evidence is obtained as the average evidence collected by batch-wise e-processes with batch sizes $[5, 10, 20]$, and over 50 seeds. As in the original experiment by \citet{shaer23a}, the model used is a lasso, and the sampler for $\widetilde X$ uses a probability $\widehat \pi(\cdot \mid Z)$ estimated through a logistic regression.
On the left we observe that all e-processes are unable to collect evidence against the null, since the mutation '12S' has no proven effect on drug resistance. 
In the middle plot, most antisymmetric e-variables are able to reject the null for a mutation affecting drug resistance (48M) but for which it is difficult to collect evidence. The GRO e-variable for exchangeability does not dominate the other antisymmetric e-variables due to derandomization and estimation error as discussed in Appendix \ref{appendix:effects_of_derandomization}. The exponential e-variable fails to reject the null: we attribute this to the inability of the constructed pseudo-density to capture the true conditional density $f_{Y|X,Z}$ of the data, that is, to the approximation error.

On the right all the methods we consider reach an evidence of at least $1/ \alpha$ and thus reject the null hypothesis for mutation '47V', which has been reported to have a major effect. As in the previous case, we explain the fact that the exponential e-variable collects the least evidence with model misspecification, even though we arguably find ourselves in the large effect regime.



\subsection{Triple robustness for safety guarantees} \label{sect:robustness_safety}

In this section we illustrate the experiment in Figure \ref{fig:safety} regarding the safety and robustness properties deriving from our bound. The setting is similar to Section \ref{subsec:sim_setting}: the data is again generated as described in Section 5 of \cite{shaer23a}, in this case without further changes. 
We generate 20000 data points, of which 200 triplets are used to pre-train the model and the sampler, while the rest is considered sequentially and is used to compute the martingales. 
The model is a multi-layer perceptron with two hidden layers of 50 nodes each, trained on a mean squared error loss. This loss is also used as a statistic for computing the sequential bets.
At each round we use the standard coin betting e-value
\begin{align}
\notag
E_{\mathrm{cb}}^\lambda
&=
1 + \lambda \left(
- l\!\left(m(X, Z), Y\right)
+ 
l\!\left(m(\widetilde{X}, Z), Y\right)
\right).
\label{eq:cb_def}
\end{align}
The squared loss is bounded by 5 and consequently the betting parameter $\lambda$ goes up to $0.2$.

We use a Gaussian sampler, which fits a multivariate Gaussian to the joint distribution of $(X, Z)$, and samples $X \mid Z$ using the closed form expression for a multivariate Gaussian conditional on some of its variables. 
The experiment is run with batch sizes $[5, 10, 20]$ and over 10 seeds. 

The goal of this experiment is to prove the triple robustness of the coin-betting e-value, as discussed in Section~\ref{sect:trpl_rob}. To do so, we show that the three quantities, namely $K_n, \widehat{m}_n$ and $\widehat{\lambda}$, converge as expected.
First, in the top left of Figure~\ref{fig:safety}, we display the average over seeds of the values of the betting martingales.
We observe that it does not reject the null hypothesis and that it is almost constant, so it will never explode. In the top right, we display the average of the betting parameters $\lambda_n$ chosen at timestep $n$ by GRAPA \citep{wang2026surenullbankruptcytestingbybetting}. Observe that it converges to 0, as seen in the three heatmaps in the bottom row, which display the average cumulative wealth obtained by a given $\lambda$, for each batch size and each interval of timesteps.
In the middle left, we display the average of the 1-Wasserstein distances between the estimated kernel $K_n(\cdot \mid Z)$ and the true kernel $K(\cdot \mid Z)$. We used $1000$ samples from each of $K_n(\cdot \mid Z)$ and $K(\cdot \mid Z)$ to approximate the 1-Wasserstein distance, and we show the discretization error by displaying the distance of $K$ to itself as a reference.
In the middle right, we display the average of the risk difference between the estimated model and the true model, $\mathbb E[\ell(\widehat m) - \ell(m_0)]$.
Thus, the middle-left and middle-right plots correspond to the terms $a_n$ and $b_n$ from Theorem~\ref{thm:antisymmetric_type_I}.

\begin{figure}
    \centering
    \includegraphics[width=1\linewidth]{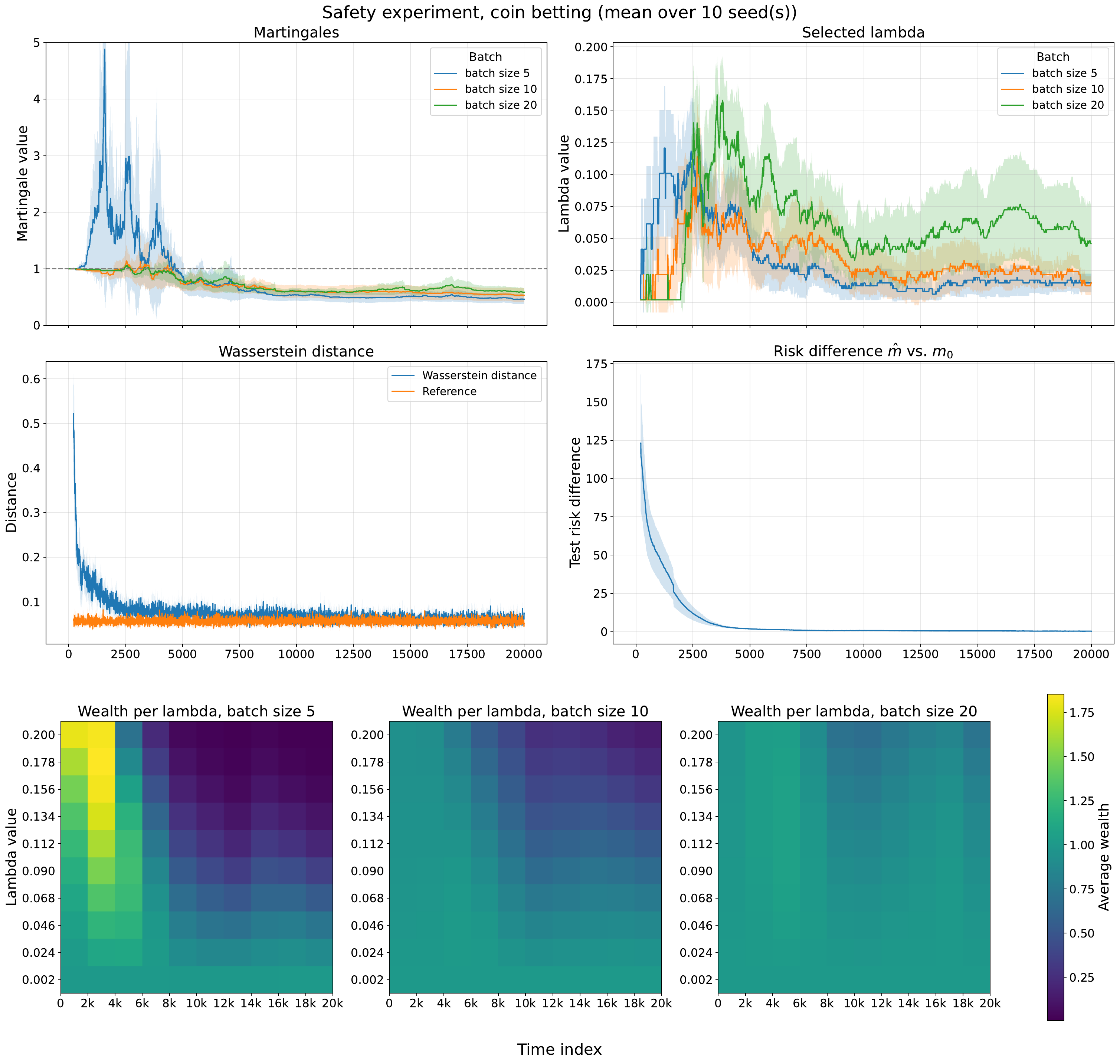}
    \caption{\textbf{Safety and robustness properties:} On the top-left the evidence collected against the null for three different batch sizes. On the top right the mean across seeds of the chosen $ \lambda$ . On the middle-left the $1$-Wasserstein distance of the estimated $K_n$ and the true conditional distribution $K$. 
    On the middle-right the value of $\mathbb E[l(\widehat m) - l(m_0)]$, where the expectation is taken over the data generating distribution and the seeds. 
    On the bottom row three heatmaps showing in each subsquare the average value of the wealth at certain timesteps for specific values of $\lambda$. We see that $\lambda \approx 0$ becomes the best value as $n$ increases.}
    \label{fig:safety}
\end{figure}

\subsection{Practical effects of derandomization}
\label{appendix:effects_of_derandomization}
The purpose of this section is to discuss the practical effects of derandomization, and highlight the fact the GRO property of $E^{\text{exch}}$ and $E^{\text{sym}}$ (see Appendix \ref{appendix:hypotheses}) generally only holds if no derandomization is performed. To eliminate the effect that learning the correct density function may have on the results of the experiment, we analyze a simple linear gaussian setting, in which the relevant densities under the alternative and GRO e-variables can be calculated explicitly.\\

We compare antisymmetric and exchangeability/symmetry GRO e-variables in the following setting. Let $X,\varepsilon \sim N(0,1)$ i.i.d., $Y=\beta X+\varepsilon$, and
fix a model estimate $\hat\beta\neq 0$. For a fresh $\widetilde{ X}\sim N(0,1)$, independent of $(X,\varepsilon)$, set
$T=(Y-\hat\beta X)^2$, $\widetilde {T}=(Y-\hat\beta\widetilde {X})^2$ and $D=\widetilde {T}-T$. In this simple setting:
\begin{itemize}
\item $Y \indep X$
\item $(T,\widetilde{T})$ is
exchangeable,
\item $D$ is symmetric,
\item $\mathbb{E}[T] = \mathbb{E}[\widetilde{T}]$, and
\item $\beta=0$
\end{itemize}
are all equivalent statements.

\paragraph{Antisymmetric e-variables.} The (non-GRO) antisymmetric e-variables considered are:
\begin{itemize}
    \item $1 + \lambda \operatorname{sign}(D)$,
    \item $1 + \lambda \tanh(D)$,
    \item $1 + \lambda \operatorname{clip}(D,-M,M)/M$.
\end{itemize}

\paragraph{The GRO e-variable(s).} Using
$U=Y-\hat\beta X = (\beta - \hat \beta)X + \varepsilon,\qquad V=Y-\hat\beta\widetilde {X}=\beta X+\varepsilon-\hat\beta\widetilde {X},$

$(U,V)$ is a \textbf{centred} bivariate normal with covariance matrix 
\begin{equation*}
   \Sigma = \left[ \begin{matrix}
\sigma_U^2 & \rho \sigma_U \sigma_V \\
\rho \sigma_U \sigma_V & \sigma_V^2
\end{matrix} \right] = \left[ \begin{matrix}
1 + (\beta-\hat\beta)^2 & 1+ \beta(\beta-\hat\beta) \\
1+ \beta(\beta-\hat\beta) & 1 + \beta^2 + \hat\beta^2
\end{matrix} \right]. 
\end{equation*}

Because $(T,\widetilde{T}) = (U^2,V^2)$ is the square of a
bivariate normal, its joint density and the density of $D$ are explicit:
\begin{align*}
f_{T,\widetilde{T}}(t,\widetilde{t}) &=  \frac{1}{2\pi\sqrt {\det \Sigma}\,\sqrt{t\widetilde{t}}}\,
\exp\!\Big(-\frac{\sigma_V^2t+\sigma_U^2\widetilde{t}}{2\det \Sigma}\Big)\cosh\!\Big(\frac{\rho \sigma_U \sigma_V\sqrt{t\widetilde{t}}}{\det \Sigma}\Big),\qquad t,\widetilde{t}>0 ,\\
f_D(d) &=\frac{1}{2\pi\sqrt {\det \Sigma}}\,\exp\!\Big(\frac{(\sigma_V^2-\sigma_U^2)\,d}{4\det \Sigma}\Big)\,K_0\!\Big(\frac{c\,|d|}{4\det \Sigma}\Big),
\end{align*}
with $c = \sqrt{(\sigma_V^2+\sigma_U^2)^2 - 4 \det \Sigma}$, and $K_0$ is the modified Bessel function of the second kind \citep{kotz2004continuous, NadarajahPogany2016}.\\
The GRO e-variables for
exchangeability and symmetry are equal, and both reduce to
\begin{equation*}
  E^{\text{exch}}(D) = \frac{2}{1+e^{-\kappa D}} =  1 + \tanh\left(\frac{\kappa D}{2}\right), \qquad \kappa= \frac{\sigma_V^2-\sigma_U^2}{2 \det \Sigma} = \frac{\beta}{\hat\beta\left(2+(\beta-\hat\beta)^2\right)}.
\end{equation*}
Exact densities remove estimation error, and  $E^{\text{exch}}(D)$ is exactly
growth-rate optimal for testing $\mathcal{P}^{\text{exch}}$ (see Appendix \ref{appendix:hypotheses}) with a single draw of $\widetilde {X}$.

\paragraph{Derandomization.} We draw $B$ independent copies $\widetilde {X}_1,\dots,\widetilde {X}_B$ per observation and
replace each e-variable by its average. For the antisymmetric family parameters $\lambda$ and $M$ can be re-tuned to derandomization, while $E^{\text{exch}}$ is optimized for $B=1$, and kept unchanged. We compare three arms at each $B$: 
\begin{itemize}
    \item antisymmetric e-variables re-tuned at $B$,
    \item antisymmetric e-variables frozen at their $B=1$ parameters -- used to isolate the effect of re-tuning,
    \item derandomized GRO.
\end{itemize}

\paragraph{Setup.} Parameters are chosen by grid search ($\lambda\in[-0.99,0.99]$, $M\in(0,10]$) on a separate
pretraining sample of size $50{,}000$ and evaluated on an independent sample of size $n=20{,}000$. We report
the e-power as mean over 10 seeds, for
correct guesses $\hat\beta=\beta$, wrong-magnitude guesses and a wrong-signed guess ($\kappa<0$). Results are in
Table~\ref{tab:derandomization} and Figure~\ref{fig:derandomization}.

\begin{table}[t]
  \centering
  \small
  \setlength{\tabcolsep}{5pt}
  \caption{E-power under the alternative, before ($B=1$) and after derandomizing over
    $B=20$ draws of $\widetilde {X}$. \emph{Frozen}  keeps the parameters tuned at $B=1$; \emph{re-tuned}
     re-tunes them at $B=20$. GRO has no parameters, so its frozen and re-tuned values coincide; both GRO versions are identical and shown once. Gains are relative to each e-variable's own
    $B=1$ e-power. Parameters are tuned on a separate pretraining sample ($50{,}000$) and e-power is measured
    on an independent sample ($n=20{,}000$); entries are means over 10 seeds, with standard errors of at most
    $0.0013$ for e-powers, $4.6$ percentage points for gains, $0.01$ for $\lambda^*$ and $0.07$ for $M^*$.
    Bold: best e-power in each setting at $B=1$ and at $B=20$.}
  \label{tab:derandomization}
  \begin{tabular}{l ccc cc cc}
    \toprule
    & \multicolumn{3}{c}{e-power} & \multicolumn{2}{c}{gain over $B=1$ (\%)} & \multicolumn{2}{c}{tuned parameters, $B=1\to20$} \\
    \cmidrule(lr){2-4} \cmidrule(lr){5-6} \cmidrule(lr){7-8}
    e-variable & $B=1$ & frozen & re-tuned & frozen & re-tuned & $\lambda^*$ & $M^*$ \\
    \midrule
    \multicolumn{8}{l}{$\beta=0.5,\ \hat\beta=0.5\quad(\kappa=0.50)$} \\
    \quad sign & 0.024 & 0.039 & 0.067 & $+62.5$ & $+181.6$ & $0.21\to0.59$ & -- \\
    \quad tanh & 0.044 & 0.066 & 0.086 & $+49.5$ & $+95.7$  & $0.48\to0.85$ & -- \\
    \quad clip & 0.053 & 0.073 & \textbf{0.089} & $+38.9$ & $+68.7$ & $0.87\to0.96$ & $4.05\to2.27$ \\
    \quad GRO  & \textbf{0.054} & \multicolumn{2}{c}{0.075} & \multicolumn{2}{c}{$+39.6$} & -- & -- \\
    \midrule
    \multicolumn{8}{l}{$\beta=1,\ \hat\beta=1\quad(\kappa=0.50)$} \\
    \quad sign & 0.079 & 0.126 & 0.221 & $+60.5$ & $+180.9$ & $0.39\to0.95$ & -- \\
    \quad tanh & 0.123 & 0.182 & \textbf{0.235} & $+48.9$ & $+92.2$ & $0.65\to0.99$ & -- \\
    \quad clip & 0.152 & 0.200 & \textbf{0.235} & $+31.1$ & $+54.5$ & $0.95\to0.99$ & $4.95\to1.75$ \\
    \quad GRO  & \textbf{0.156} & \multicolumn{2}{c}{0.205} & \multicolumn{2}{c}{$+32.0$} & -- & -- \\
    \midrule
    \multicolumn{8}{l}{$\beta=1,\ \hat\beta=0.5\quad(\kappa=0.89)$} \\
    \quad sign & 0.073 & 0.116 & 0.207 & $+59.2$ & $+183.8$ & $0.37\to0.94$ & -- \\
    \quad tanh & 0.131 & 0.189 & 0.226 & $+43.9$ & $+71.9$  & $0.77\to0.99$ & -- \\
    \quad clip & 0.143 & 0.186 & \textbf{0.230} & $+30.1$ & $+60.9$ & $0.94\to0.99$ & $2.73\to0.97$ \\
    \quad GRO  & \textbf{0.146} & \multicolumn{2}{c}{0.192} & \multicolumn{2}{c}{$+30.8$} & -- & -- \\
    \midrule
    \multicolumn{8}{l}{$\beta=1,\ \hat\beta=2\quad(\kappa=0.17)$} \\
    \quad sign & 0.056 & 0.093 & 0.172 & $+66.1$ & $+206.9$ & $0.33\to0.91$ & -- \\
    \quad tanh & 0.074 & 0.117 & 0.178 & $+59.1$ & $+142.5$ & $0.44\to0.93$ & -- \\
    \quad clip & 0.112 & 0.160 & \textbf{0.184} & $+42.4$ & $+64.0$ & $0.83\to0.97$ & $10.00\to5.30$ \\
    \quad GRO  & \textbf{0.117} & \multicolumn{2}{c}{0.161} & \multicolumn{2}{c}{$+37.3$} & -- & -- \\
    \midrule
    \multicolumn{8}{l}{$\beta=2,\ \hat\beta=1\quad(\kappa=0.67)$} \\
    \quad sign & 0.158 & 0.246 & 0.392 & $+56.3$ & $+148.6$ & $0.55\to0.99$ & -- \\
    \quad tanh & 0.237 & 0.335 & 0.386 & $+41.3$ & $+63.2$  & $0.82\to0.99$ & -- \\
    \quad clip & 0.272 & 0.332 & \textbf{0.394} & $+22.1$ & $+44.9$ & $0.98\to0.99$ & $4.15\to0.50$ \\
    \quad GRO  & \textbf{0.278} & \multicolumn{2}{c}{0.342} & \multicolumn{2}{c}{$+23.2$} & -- & -- \\
    \midrule
    \multicolumn{8}{l}{$\beta=-1,\ \hat\beta=1\quad(\kappa=-0.17)$} \\
    \quad sign & 0.031 & 0.050 & 0.081 & $+61.7$ & $+164.5$ & $-0.25\to-0.61$ & -- \\
    \quad tanh & 0.044 & 0.067 & 0.107 & $+54.1$ & $+147.3$ & $-0.35\to-0.90$ & -- \\
    \quad clip & 0.068 & 0.093 & \textbf{0.128} & $+36.6$ & $+87.3$ & $-0.79\to-0.99$ & $10.00\to4.25$ \\
    \quad GRO  & \textbf{0.070} & \multicolumn{2}{c}{0.093} & \multicolumn{2}{c}{$+33.2$} & -- & -- \\
    \bottomrule
  \end{tabular}
\end{table}

\begin{figure}[h]
    \centering
    \includegraphics[width=1.0\textwidth]{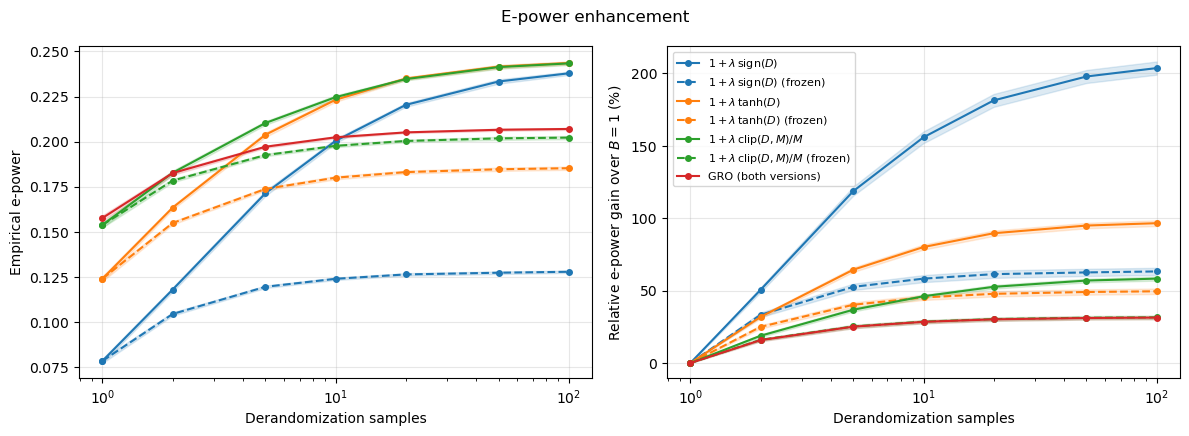}
    \caption{E-power as a function of the number of derandomization samples with $(\beta, \hat \beta) = (1,1)$, under the same setup as in Table \ref{tab:derandomization}. Full lines represent re-tuned, while dotted represent frozen parameter e-variables.}
    \label{fig:derandomization}
\end{figure}

The results clearly indicate that the GRO property of $E^{\text{exch}}$ holds at $B=1$, but is not preserved after derandomization. Furthermore, the e-power increase of antisymmetric e-variables mainly comes from the adaptation of the betting parameters. Larger $B$ implies more aggressive betting is desirable, i.e. a larger  optimal $\lambda$. Since the GRO e-variable was specifically designed for optimal betting with $B=1$, it does not have any free parameters left which could allow it to learn to make riskier bets. In practice, the GRO e-variable would be estimated via kernel density estimation methods whose parameters such as the kernel family and bandwidth could be chosen to optimize e-power instead of density estimation, and thus possibly yield a worse density approximation but larger e-power for $B>1$. We leave such questions for future research.

\end{document}